\documentclass[11pt]{article}
\usepackage{amsmath,amsthm,amssymb}
\usepackage{color}

\usepackage[hyperfootnotes=false]{hyperref}

\usepackage{pxfonts}
\usepackage{mathrsfs}

\usepackage{titlesec}
\titleformat{\paragraph}
{\normalfont\normalsize\bfseries}{\theparagraph}{1em}{}
\titlespacing*{\paragraph}
{0pt}{3.25ex plus 1ex minus .2ex}{1.5ex plus .2ex}

\def\titlerunning#1{\gdef\titrun{#1}}
\makeatletter
\def\author#1{\gdef\autrun{\def\and{\unskip, }#1}\gdef\@author{#1}}
\def\address#1{{\def\and{\\\hspace*{18pt}}\renewcommand{\thefootnote}{}%
\footnote {#1}}%
\markboth{\autrun}{\titrun}}
\makeatother
\def\email#1{\hspace*{4pt}{\em e-mail}: #1}
\def\MSC#1{{\renewcommand{\thefootnote}{}%
\footnote{\emph{Mathematics Subject Classification (2020):} #1}}}
\def\keywords#1{\par\medskip
\noindent\textbf{Keywords:} #1}

\newtheorem{theorem}{Theorem}[section]
\newtheorem{prop}[theorem]{Proposition}
\newtheorem{cor}[theorem]{Corollary}
\newtheorem{lemma}[theorem]{Lemma}

\newtheorem{defin}[theorem]{Definition}

\theoremstyle{definition}

\newtheorem{remark}[theorem]{Remark}
\newtheorem{exa}[theorem]{Example}

\numberwithin{equation}{section}

\def\0{\mathbf 0}

\def\eps{\varepsilon}

\def\cC{\mathscr C}

\def\cI{\mathcal I}
\def\cJ{\mathcal J}
\def\cL{\mathcal L}

\def\cP{\mathcal P}
\def\cQ{\mathcal Q}

\def\cD{\mathcal D}
\def\cU{\mathcal U}

\def\cT{\mathcal T}
\def\cR{\mathcal R}
\def\cS{\mathcal S}

\def\cX{\mathcal X}
\def\cY{\mathcal Y}
\def\cZ{\mathcal Z}

\def\PG{{\rm PG}}

\def\F{{\mathbb F}}

\def\GL{{\rm GL}}

\def\PGaL{{\rm P\Gamma L}}

\begin{document}

%%%%% To ease editing, add:

\baselineskip=16pt

%%%%%%%%%%%%%%%%

%% In the running head, give an abbreviation of the title.
\titlerunning{}

\title{On line-parallelisms of $\PG(3, q)$}

\author{Francesco Pavese \and Paolo Santonastaso}

\date{}

\maketitle

\address{F. Pavese, P. Santonastaso: Dipartimento di Meccanica, Matematica e Management, Politecnico di Bari, Via Orabona 4, 70125 Bari, Italy; \email{francesco.pavese@poliba.it, paolo.santonastaso@poliba.it}
}

\bigskip

\MSC{Primary 51E23; Secondary 05B40; 05B25}

%%%%%%%%

\begin{abstract}
Let $\PG(3, q)$ denote the three-dimensional projective space over the finite field with $q$ elements. A {\em line-spread} of $\PG(3, q)$ is a collection $\cS$ of mutually skew lines such that every point of $\PG(3, q)$ lies on exactly one line of $\cS$. A {\em parallelism} of $\PG(3, q)$ is a set $\Pi$ of mutually skew line-spreads of $\PG(3, q)$ such that every line of $\PG(3, q)$ is contained in precisely one line-spread of $\Pi$. %A {\em regulus} of $\PG(3,q)$ is the set of transversals to three pairwise skew lines, and consists of $q +1$ pairwise skew lines. A spread $\cS$ is {\em regular} or {\em Desarguesian} if each line not in $\cS$ meets $\cS$ in the lines of a regulus. A {\em Hall spread} is obtained from a Desarguesian spread by switching a regulus with its opposite  regulus. 
For a Desarguesian spread $\cD$ and an elementary abelian group $E$ of order $q^2$ that stabilizes $\cD$ and one of its lines, let $\cT$ be the class of parallelisms of $\PG(3, q)$ admitting $E$, and comprising $\cD$ and $q^2+q$ Hall spreads, each of which is obtained by switching one of the $q^2+q$ reguli of $\cD$ through its $E$-fixed line. 
In this paper, the parallelisms in $\cT$ are characterized geometrically and enumerated. Moreover, it is shown that $\cT$ contains at least $\Theta(q^{q-1} q!)$ mutually inequivalent parallelisms for $q$ even, and at least $\Theta(q^{2q-3})$ mutually inequivalent parallelisms when $q$ is odd.

\end{abstract}

\section{Introduction}

Let $\PG(3, q)$ denote the three-dimensional projective space over the finite field with $q$ elements. A {\em line-spread} of $\PG(3, q)$ is a collection $\cS$ of mutually skew lines such that every point of $\PG(3, q)$ lies on exactly one line of $\cS$. A {\em parallelism} of $\PG(3, q)$ is a set $\Pi$ of mutually skew line-spreads of $\PG(3, q)$ such that every line of $\PG(3, q)$ is contained in precisely one line-spread of $\Pi$. A {\em regulus} of $\PG(3,q)$ is the set of transversals to three pairwise skew lines, and consists of $q +1$ pairwise skew lines. A spread $\cS$ is {\em regular} or {\em Desarguesian} if each line not in $\cS$ meets $\cS$ in the lines of a regulus. A {\em Hall spread} is obtained from a Desarguesian spread by switching a regulus with its opposite  regulus. The first examples of parallelisms of $\PG(3, q)$ were obtained by Denniston \cite{D1972} and Beutelspacher \cite{beutelspacher1974parallelisms}. Infinite families of parallelisms of $\PG(3, q)$ consisting of regular spreads were constructed by Penttila and Williams in \cite{PW1998} in the case when $q \equiv 2 \pmod{3}$. Many authors have investigated and exhibited construction techniques for parallelisms of $\PG(3, q)$, see e.g. \cite{J2000, J2003, J2008}. For further results, the reader is referred to \cite{J2010} and references therein. Some sporadic examples obtained with the aid of a computer can be found in \cite{Prince1998, Betten, Topalova}. Most of the known infinite families of parallelisms have one Desarguesian spread $\cD$ and admit an elementary abelian group $E$ of order $q^2$ that stabilizes $\cD$ and one of its lines. In the case when $q$ is even, these parallelisms comprise the Desarguesian spread $\cD$ and $q^2+q$ Hall spreads, each of which is obtained by switching one of the $q^2+q$ reguli of $\cD$ through its $E$-fixed line. Other examples are obtained from parallelisms of this type by means of a derivation process \cite[Theorem 175]{J2010}.  

\subsection{Our Contribution}
For a Desarguesian spread $\cD$ of $\PG(3,q)$, let $E$ be an elementary abelian group of order $q^2$ that stabilizes $\cD$ and one of its lines. Consider the class $\cT$ of parallelisms of $\PG(3, q)$ admitting $E$ as an automorphism group, and comprising $\cD$ and $q^2+q$ Hall spreads, each of which is obtained by switching one of the $q^2+q$ reguli of $\cD$ through its $E$-fixed line. In this paper, we present a geometric method that allows us to characterize and enumerate all the parallelisms in $\cT$. Precisely, we prove that there is a correspondence between certain sets of $q+1$ Bear pencils within the larger space $\PG(3,q^2)$, referred to as \emph{good sets}, and parallelisms in $\cT$: we establish that every good set determines a parallelism, and conversely, each parallelism can be described explicitly by a corresponding good set.
Finally, using this combinatorial-geometric framework, it is shown that the class $\cT$ contains at least $\Theta(q^{q-1} q!)$ mutually inequivalent parallelisms for $q$ even, and at least $\Theta(q^{2q-3})$ mutually inequivalent parallelisms when $q$ is odd.

\section{Preliminary results}

Let $\F_{q^n}$ be the finite field of order $q^n$, where $q = p^m$ and $p$ is a prime. We denote by $\PG(r, q^n)$ the $r$-dimensional projective space over $\F_{q^n}$, equipped with homogeneous projective coordinates $(x_1, \dots, x_{r+1})$. The $q$-order subgeometries of $\PG(3,q^2)$ are also referred to as \emph{Baer subgeometries}, and their subspaces as \emph{Baer subspaces}.
%Let $N$ denote the norm function, defined as $N(x) = x^{q+1}$, where $x \in \F_{q^2}$. 
We will represent projectivities of $\PG(3, q^2)$ with elements in $\GL(4, q^2)$, with matrices acting on the left.

A notable example of line-spread of $\PG(3, q)$ is the {\em Desarguesian line-spreads}, which can be constructed as follows; see \cite{segre1964teoria}. Let $\Sigma \simeq \PG(3,q)$ be a subgeometry of $ \PG(3,q^2)$, such that $\Sigma$ is the set of fixed points of a semilinear collineation $\tau$ of $\PG(3,q^2)$ of order $2$. Let $\ell$ be a line of $\PG(3,q^2)$ such that $\ell$ and $\ell^{\tau}$ are skew, or equivalently, they span $\PG(3,q^2)$. For each point $P$ of $\ell$, define the line $r_P=\langle P,P^{\tau}\rangle_{q^2}$ of $\PG(3,q^2)$. 

\begin{lemma} [see \textnormal{\cite[Lemma 1]{lunardon1999normal}}] \label{lm:subspacesubgeo} Let $S$ be a subspace of $\PG(3,q^2)$. Then $\dim(S \cap \Sigma)=\dim(S)$ if and only if $S^{\tau}=S$. 
\end{lemma}
Note that $r_{P}$ is fixed by $\tau$ and so, by Lemma \ref{lm:subspacesubgeo} intersects $\Sigma$ in a subline. Moreover, two lines $r_P$ and $r_{P'}$, with $P, P' \in \ell$, $P \ne P'$, are disjoint, otherwise the subspace $\langle r_P, r_{P'} \rangle_{q^2}$ is a plane, but this is impossible since the lines $\ell$ and $\ell^{\tau}$ are disjoint. it follows that 
\[
\mathcal{D}=\{r_P \cap \Sigma \colon P \in \ell\}
\]
is a set of $q^2 + 1$ pairwise skew lines of $\PG(3,q)$. In other words, $\mathcal{D}$ forms a line spread of $\Sigma$, known as a \emph{Desarguesian line-spread}. The lines $\ell$ and $\ell^{\tau}$ are called the {\em transversal lines} or \emph{director spaces} of $\mathcal{D}$; see also \cite{lunardon1999normal}.

\subsection{The geometric setting}

%Let $\F_{q^2}$ be the finite field of order $q^2$, where $q = p^h$ and $p$ is a prime, and let $N$ denote the norm function, defined as $N(x) = x^{q+1}$, where $x \in \F_{q^2}$. Let $\PG(3, q^2)$ be the three-dimensional projective space over $\F_{q^2}$ where $(X_1, X_2, X_3 , X_{4})$ are projective homogeneous coordinates. We will represent projectivities of $\PG(3, q^2)$ with elements in $\GL(4, q^2)$, with matrices acting on the left. 
Let $U_i$ denote the point of $\PG(3,q^2)$ having $1$ in the $i$-th position and $0$ elsewhere, $1 \le i \le 4$, and let $t_1$, $t_2$ be the lines $\langle U_1, U_{2} \rangle_{q^2}$, $\langle U_{3}, U_{4} \rangle_{q^2}$, respectively. Let $\alpha \in \F_{q^2} \setminus \{ 0 \}$ and let $\tau_{\alpha}$ be the non-linear involution of $\PG(3,q^2)$ mapping the point $\left(x_1, x_2, x_3, x_{4}\right)$ to the point $\left(x_{3}^q, x_{4}^q, \alpha^{q+1} x_1^q, \alpha^{q+1} x_{2}^q\right)$. Then $\tau_{\alpha}$ is the composition of the projectivity represented by the matrix 
\begin{align*}
\begin{pmatrix}
0 & 0 & 1 & 0 \\
0 & 0 & 0 & 1 \\
\alpha^{q+1} & 0 & 0 & 0 \\
0 & \alpha^{q+1} & 0 & 0 
\end{pmatrix}
\end{align*}
with the non-linear involution sending the point $P = \left(x_1, x_2, x_3 , x_{4}\right)$ to the point $P^q = \left(x_1^q, x_2^q, x_3^q, x_{4}^q\right)$. The set of points fixed by $\tau_{\alpha}$ forms a Baer subgeometry $\Sigma_{\alpha}$ isomorphic to $\PG(3, q)$. In particular the points of $\Sigma_{\alpha}$ have homogeneous projective coordinates 
\begin{align*}
& \left(x, y, \alpha x^q, \alpha y^q\right), & \mbox{ where } x, y \in \F_{q^2} \mbox{ and } (x, y) \ne (0, 0). 
\end{align*}
Let $\beta \in \F_{q^2} \setminus \{ 0 \}$ be such that $\alpha \ne \beta$. Denote by $\Sigma_{\beta}$ the Baer subgeometry of $\PG(3,q^2)$ fixed pointwise by the non-linear involution $\tau_{\beta}$. Clearly, $\Sigma_{\alpha} = \Sigma_{\beta}$ if and only if $\alpha^{q+1} = \beta^{q+1}$. On the other hand, for any point $P \in t_1$, we have that $P^{\tau_{\alpha}} = P^{\tau_{\beta}}$ and, if $\alpha^{q+1} \ne \beta^{q+1}$, then the Baer subgeometry $\Sigma_{\alpha}$ is disjoint from $\Sigma_{\beta}$. 

For a subset $A \subseteq \F_q \setminus \{0\}$, we denote by $A^{-1}:=\left\{a^{-1}\colon a \in A\right\}$.

\begin{lemma} \label{lm:partitionnorm}
    Assume that $q \geq 5$ is odd. Then there exists a subset $A=\left\{a_1,\ldots,a_\frac{q-3}{2}\right\}$ of $\F_q \setminus \{0\}$ such that 
    \begin{itemize}
        \item $1,-1 \notin A$, 
        \item $-a_i \notin A$, for every $i \in \{1,\ldots,t\}$,
        \item $A \cap A^{-1}=\emptyset$ and
    \[
        \F_q \setminus \{0\} = \{\pm 1\} \cup A \cup A^{-1}.
        \]
        \end{itemize}
    %Moreover $t=\frac{q-3}{2}$.
\end{lemma}
\begin{proof}
    Recall that $-1$ is a square in $\F_q$ if and only if $q \equiv 1 \pmod{4}$. \\
    First, assume that $q \equiv 1 \pmod{4}$. Then there exists $b \in \F_q \setminus \{0\}$ with $b^2 = -1$. Clearly, $b^{-1} = -b$. If $q = 5$, we set $A = \{b\}$ and cleary $A$ has the desired properties. So, let $q > 5$. For any $c \in \F_q^* \setminus \{0, \pm 1, \pm b\}$, the elements $c, -c, c^{-1}, -c^{-1}$ are all distinct. Thus, we have a partition:
    \[
    \{\pm 1\} \cup \{b,-b\} \cup \{c_1,-c_1,1/c_1,-1/c_1\} \cup \ldots \cup \{c_h,-c_h,1/c_h,-1/c_h\} = \F_q \setminus \{0\}, 
    \]
  with distinct $c_i \in \F_q \setminus \{0\}$, where $h = \frac{q - 5}{4}$. By defining
    \[
    A=\{b,c_1,-1/c_1,\ldots,c_h,-1/c_h\}
    \]
    we see that $A$ satisfies all required conditions, and its size is indeed $2h+1=\frac{q-3}{2}$.\\
   Now, suppose $q \equiv 3 \pmod{4}$. In this case, $-1$ is not a square, hence no element satisfies $b^2 = -1$. For any $c \in \F_q \setminus \{0, \pm 1\}$, the elements $c, -c, c^{-1}, -c^{-1}$ remain distinct. Thus, we have the partition:
    \[
    \{\pm 1\} \cup \{c_1,-c_1,1/c_1,-1/c_1\} \cup \ldots \cup \{c_h,-c_h,1/c_h,-1/c_h\} = \F_q \setminus \{0\}, 
    \]
    with distinct $c_i \in \F_q \setminus \{0\}$, where $h = \frac{q - 3}{4}$. The set 
    \[
    A=\{c_1,-1/c_1,\ldots,c_h,-1/c_h\}
    \]
    clearly fulfills all conditions, and its cardinality is $2h=\frac{q-3}{2}$.
\end{proof}
Assume throughout the paper that $q \ge 3$ and fix a set $\Lambda$ of $q-1$ elements in $\F_{q^2} \setminus \{0\}$ having pairwise distinct norm. In this way, as $\alpha$ ranges over the set $\Lambda$, we obtain $q - 1$ mutually disjoint subgeometries $\Sigma_\alpha$.
We fix the following partition of $\F_q \setminus \{0\}$ if $q$ is even
\begin{equation} \label{eq:definitionaieven} \left\{ 1 \right\} \cup \left\{a_1, \dots, a_t \right\} \cup \left\{a_1^{-1}, \dots, a_t^{-1} \right\} = \F_{q} \setminus \{0\},
\end{equation}
while if $q$ is odd
\begin{equation} \label{eq:definitionaiodd} \left\{ \pm 1 \right\} \cup \left\{a_1, \dots, a_t \right\} \cup \left\{a_1^{-1}, \dots, a_t^{-1} \right\} = \F_{q} \setminus \{0\},
\end{equation}
where the set $\{a_1,\ldots,a_t\}$ is as in Lemma \ref{lm:partitionnorm}, in the case where $q \ge 5$. Here,
\begin{align*}
    & t = 
    \begin{cases}
        \frac{q-2}{2} & \mbox{ if } q \mbox{ is even}, \\
        \frac{q-3}{2} & \mbox{ if } q \mbox{ is odd}.
    \end{cases}
\end{align*}
Let 
\begin{align*}
\cU = \{u \in \F_{q^2} \mid u^{q+1} = 1\}
\end{align*}
and set 
\begin{align*}
    & \eta \in \cU \cap \Lambda, \\
    & \cI := \left\{\alpha \in \Lambda \mid \alpha^{q+1} \in \left\{a_1, \dots, a_t \right\} \right\}, & \mbox{ if } q \mbox{ is even, } \\
    & \cI := 
    \left\{\alpha \in \Lambda \mid \alpha^{q+1} \in \left\{a_1, \dots, a_t \right\} \cup \{-1\} \right\}, & \mbox{ if } q \mbox{ is odd. }
\end{align*}

Observe that 
\begin{align*}
    & |\cI| = \begin{cases}
    \frac{q-2}{2} & \mbox{ if } q \mbox{ is even}, \\
        \frac{q-1}{2} & \mbox{ if } q \mbox{ is odd}, 
    \end{cases} \\
    & \eta \not\in \cI. 
    %& \alpha \not\in \cI, \mbox{ in the case when } q \mbox{ is odd and } \alpha^{q+1} = -1.
\end{align*}
Also, if $q$ is odd 
\begin{equation} \label{eq:oppositenonbelong}
    \alpha \in \cI \Rightarrow \beta \notin \cI, \mbox{ where }\beta^{q+1}=-\alpha^{q+1}.
\end{equation}

A straightforward calculation shows what follows.
\begin{lemma}\label{lemmasub}
$\Sigma_{\alpha}^{\tau_{\eta}} = \Sigma_{\beta}$, where $\beta^{q+1} = \frac{1}{\alpha^{q+1}}$.
\end{lemma}
Therefore, $\Sigma_{\eta}$ is fixed pointwise by $\tau_{\eta}$. If $q$ is odd, $\tau_{\eta}$ fixes $\Sigma_{\alpha}$, with $\alpha^{q+1} = -1$, setwise, and permutes in pairs the remaining $q-3$ Baer subgeometries, whereas if $q$ is even, the $q-2$ Baer subgeometries distinct from $\Sigma_{\eta}$ are paired by $\tau_{\eta}$. Moreover, for $\alpha,\beta \in \Lambda \setminus \cU$, and $\alpha^{q+1} \neq -1,\beta^{q+1} \neq -1$ such that $\Sigma_{\alpha}^{\tau_{\eta}} = \Sigma_{\beta}$, we have 
\begin{align}
    \alpha \in \cI \iff \beta \not\in \cI. \label{indici}
\end{align}

For each $\alpha \in \Lambda$, the spaces $t_1$, $t_2$ are disjoint from $\Sigma_{\alpha}$ and $t_1^{\tau_{\alpha}} = t_2$. For every point $P \in t_1$, let 
\begin{align*}
r_P = \langle P, P^{\tau_{\alpha}} \rangle_{q^2}
\end{align*}
be the line joining the points $P$ and $P^{\tau_{\alpha}}$. Then $r_{P}$ is fixed by $\tau_{\alpha}$ and intersects $\Sigma_{\alpha}$ in a Baer subline. Hence 
\begin{align*}
\cD_{\alpha} = \{ r_P \cap \Sigma_{\alpha} \;\; | \;\; P \in t_1 \}
\end{align*} 
is a Desarguesian line-spread of $\Sigma_{\alpha}$, for each $\alpha \in \Lambda$, and the lines $t_1$, $t_2$ are the transversal lines to $\cD_{\alpha}$.
For a set $\cX$ consisting of (Baer) lines of $\Sigma_{\eta}$, we denote by $\overline{\cX}$ the union of the lines of $\cX$ extended over $\F_{q^2}$, i.e., 
\begin{align*}
& \overline{\cX} = \bigcup_{\substack{r \mbox{\scriptsize{ line of }} \PG(3, q^2) \\ r \cap \Sigma_{\eta} \in \cX}} r.
\end{align*}
Therefore 
\begin{align*}
\overline{\cD}_{\eta} = \bigcup_{P \in t_1} r_{P} = \bigcup_{\alpha \in \Lambda} \Sigma_{\alpha} \cup t_1 \cup t_2.
\end{align*}

\begin{lemma}\label{lemma1}
Let $\alpha, \beta \in \Lambda$. The following hold.
\begin{itemize}
\item[i)] A plane $\pi$ of $\PG(3, q^2)$ such that $\pi \cap \Sigma_{\alpha}$ is a Baer subplane meets $\Sigma_{\beta}$, $\alpha \ne \beta$, in a Baer subline belonging to $\cD_{\beta}$.
\item[ii)] A line $\ell$ of $\PG(3, q^2)$ such that $\ell \cap \Sigma_{\alpha}$ is a Baer subline not belonging to $\cD_{\alpha}$ is disjoint from $\Sigma_{\beta}$, $\alpha \ne \beta$.
\end{itemize}
\end{lemma}
\begin{proof}
Since $\pi \cap \Sigma_{\alpha}$ is a Baer subplane, by Lemma \ref{lm:subspacesubgeo}, it follows that $\pi$ is fixed by $\tau_{\alpha}$. On the other hand, let $P \in \pi \cap t_1$. Then $P^{\tau_{\alpha}} \in \pi^{\tau_{\alpha}} = \pi$, which implies that the line $r_P = \langle P, P^{\tau_{\alpha}} \rangle_{q^2}$ is contained in $\pi$. Hence, for every $\beta \in \Lambda$, the Baer subline $r_P \cap \Sigma_{\beta} \in \cD_{\beta}$ is contained in $\pi \cap \Sigma_{\beta}$. Moreover, if $\alpha \ne \beta$, the number of planes such that $\pi \cap \Sigma_{\alpha}$ and $\pi \cap \Sigma_{\beta}$ are Baer subplanes equals the number of points belonging to both $\Sigma_{\alpha}$, $\Sigma_{\beta}$, see for instance \cite{bruen1982intersection}. Therefore, $r_P \cap \Sigma_{\beta} = \pi \cap \Sigma_{\beta}$, since $|\Sigma_{\alpha} \cap \Sigma_{\beta}| = 0$. This proves the first part. 

%Since $\pi \cap \Sigma_{\alpha_i}$ is a Baer subplane, and $\cD_i$ is a line-spread of $\Sigma_{\alpha_i}$, there exists $P \in t_1$ such that the Baer subline $r_{P} \cap \Sigma_{\alpha_i}$ is contained in  $\pi \cap \Sigma_{\alpha_i}$. Hence the Baer subline $r_P \cap \Sigma_{\alpha_j} \in \cD_j$ is contained in $\pi \cap \Sigma_{\alpha_j}$. Moreover, if $i \ne j$, the number of planes such that $\pi \cap \Sigma_{\alpha_i}$ and $\pi \cap \Sigma_{\alpha_j}$ are Baer subplanes equals the number of points belonging to both $\Sigma_{\alpha_i}$, $\Sigma_{\alpha_j}$, see for instance \cite{bruen1982intersection}. Therefore $r_P \cap \Sigma_{\alpha_j} = \pi \cap \Sigma_{\alpha_j}$, since $|\Sigma_{\alpha_i} \cap \Sigma_{\alpha_j}| = 0$. This proves the first part. 

Let $\ell$ be a line of $\PG(3, q^2)$ such that $\ell \cap \Sigma_{\alpha}$ is a Baer subline not belonging to $\cD_{\alpha}$ and let $\pi$ be a plane such that $\pi \cap \Sigma_{\alpha}$ is a Baer subplane with $\ell \cap \Sigma_{\alpha} \subset \pi \cap \Sigma_{\alpha}$. Then, if $\alpha \ne \beta$, we have just seen that $\pi \cap \Sigma_{\beta} = r_{P} \cap \Sigma_{\beta}$, for some $P \in t_1$. Therefore $|\ell \cap \Sigma_{\beta}| = 0$ 
\end{proof} 

\begin{lemma}\label{lemma2}
The following hold.
\begin{itemize}
\item[i)] $|\overline{\cD}_{\eta}| = (q^2+1)^2$.
\item[ii)] Let $\pi$ be a plane of $\PG(3, q^2)$, then $|\pi \cap \overline{\cD}_{\eta}| \in \{q^2+1, 2q^2+1\}$. If $|\pi \cap \overline{\cD}_{\eta}| = 2q^2+1$, then $r_P \subset \pi$, for some $P \in t_1$ and either $\pi \cap \overline{\cD}_{\eta} = r_{P} \cup t_i$, where $i \in \{1,2\}$, or $\pi \cap \overline{\cD}_{\eta} = r_P \cup \left( \pi \cap \Sigma_{\alpha} \right)$, where $\alpha \in \Lambda$. 
\end{itemize}
\end{lemma}
\begin{proof}
Since $\overline{\cD}_{\eta}$ is the union of the $q^2+1$ pairwise disjoint lines in $\{r_{P} \mid P \in t_1\}$, it follows that $|\overline{\cD}_{\eta}| = (q^2+1)^2$. Let $\pi$ be a plane of $\PG(3,q^2)$. If $r_P \cap \pi$ is a point, for every $P \in t_1$, then $|\pi \cap \overline{\cD}_{\eta}| = q^2+1$. Otherwise, there is a line $r_P$ contained in $\pi$. Moreover, by Lemma~\ref{lemma1}, {\em i)}, among the $q^2+1$ planes through $r_P$ there are $q^2-1$ of them intersecting precisely one subgeometry $\Sigma_{\alpha}$ in a Baer subplane. Therefore either $\pi \cap \overline{\cD}_{\eta} = r_P \cup t_i$, for some $i \in \{1, 2\}$, or $\pi \cap \overline{\cD}_{\eta} = r_P \cup \left( \pi \cap \Sigma_{\alpha} \right)$, for some $\alpha \in \Lambda$. 
\end{proof}

Consider the line 
\begin{align*}
r_{U_1} = \langle U_1, U_1^{\tau_{\alpha}} \rangle_{q^2} = \langle U_1, U_3 \rangle_{q^2}.
\end{align*}

\begin{lemma}\label{prop0} 
    Let $\cR$ be a regulus of $\cD_{\eta}$ containing $r_{U_1} \cap \Sigma_{\eta}$. If $\ell$ is a line of $\PG(3, q^2)$ having one point in common with each of the $q+1$ lines of $\cR$ extended over $\F_{q^2}$, then either $\ell \in \{t_1, t_2\}$, or $\ell \cap \Sigma_{\alpha}$ is a Baer subline, for some $\alpha \in \Lambda$, not belonging to $\cD_{\alpha}$. Moreover, $\ell$ intersects $r_{U_1}$ at exactly one point. 
\end{lemma}
\begin{proof}
Denote by $\tilde{\cR}$ the regulus of $\PG(3, q^2)$ such that $\left\{r \cap \Sigma_{\eta} \mid r \in \tilde{\cR}\right\}$ coincides with $\cR$ and by $\tilde{\cR}^o$ the opposite regulus of $\tilde{\cR}$. If $\ell$ is a line of $\PG(3, q^2)$ having one point in common with each of the $q+1$ lines of $\cR$ extended over $\F_{q^2}$, then $\ell$ intersects $r_{U_1}$ at exactly one point and $\ell \in \tilde{\cR}^o$. Note that $t_1, t_2 \in \tilde{\cR}^o$ and for every $\alpha \in \Lambda$, there are $q+1$ lines of $\tilde{\cR}^o$ intersecting $\Sigma_{\alpha}$ in a Baer subline not belonging to $\cD_{\alpha}$. Since $\left|\tilde{\cR}^o\right| = q^2+1$, by Lemma~\ref{lemma1}, {\em ii)}, it follows that every line of $\tilde{\cR}^o \setminus \{t_1, t_2\}$ meets $\Sigma_{\alpha}$, for some $\alpha \in \Lambda$, in a Baer subline not belonging to $\cD_{\alpha}$.    
\end{proof}

For an element $\alpha \in \cI$, we have
\begin{align*}
& r_{U_1} \cap \Sigma_{\alpha} = \left\{P_{\alpha u} \mid u \in \cU\right\}, & \mbox{ with } P_{\alpha u} = (1,0,\alpha u, 0),
\end{align*}
is a Baer subline. Similarly, each of the $q+1$ planes $\pi_{\alpha v}$ through $r_{U_1}$ 
\begin{align*}
& \pi_{\alpha v}: X_4 = \alpha v X_2, & v \in \cU,
\end{align*}
meets $\Sigma_{\alpha}$ in a Baer subplane. Let 
\begin{align*}
p\left(P_{\alpha u}, \pi_{\alpha v}\right)
\end{align*}
be the Baer subpencil formed by the $q+1$ lines through $P_{\alpha u}$ contained in $\pi_{\alpha v}$ and meeting $\pi_{\alpha v} \cap \Sigma_{\alpha}$ in a Baer subline. Note that $r_{U_1} \in p\left(P_{\alpha u}, \pi_{\alpha v}\right)$, $\forall \alpha \in \cI$, $\forall u, v \in \cU$. Set
\begin{align*}
& \cL = \left\{p\left(P_{\alpha u}, \pi_{\alpha v}\right) \setminus \left\{r_{U_1}\right\} \mid \alpha \in \cI, u, v \in \cU \right\}.
\end{align*} 
Hence $\cL$ is the set of $|\cI| q(q+1)^2$ lines of $\PG(3, q^2)$ intersecting $r_{U_1}$ in a point of $\Sigma_{\alpha}$, contained in $\pi_{\alpha v}$, for some $v \in \cU$, and meeting $\pi_{\alpha v} \cap \Sigma_{\alpha}$ in a Baer subline, where $\alpha \in \cI$.

%\textcolor{red}{Ho modificato il lemma 2.6, ma nella sua nuova versione occorre aggiungere  ad $\cI$, nel caso $q$ dispari, il valore $\alpha_k$, dove $\alpha_k^{q+1} = -1$.}
\begin{lemma}\label{lemma3}
The lines $\ell$, $\ell^{\tau_{\eta}}$ are the transversals of a Desarguesian line-spread $\cS_\ell$ of $\Sigma_{\eta}$ such that $\cD_{\eta} \cap \cS_\ell$ is a regulus through $r_{U_1} \cap \Sigma_{\eta}$ if and only if at least
one among $\ell$ and $\ell^{\tau_\eta}$ belongs to $\cL$.
\end{lemma}
\begin{proof}
Assume that $\ell \in \cL$. By Lemma~\ref{lemma1}, {\em ii)}, $|\ell \cap \Sigma_{\eta}| = 0$. Therefore, we easily infer that $|\ell^{\tau_{\eta}} \cap \Sigma_{\eta}| = 0$, being $\tau_{\eta}$ an involution such that $Fix(\tau_{\eta}) = \Sigma_{\eta}$. This shows that $\cS_\ell$ is a Desarguesian line-spread  of $\Sigma_{\eta}$. The line $r_{P}$, with $P \in t_1$, $P \ne U_1$, meets $\pi_{\alpha v}$ at a point of $\Sigma_{\alpha} \setminus r_{U_1}$ and through a point of $\left(\pi_{\alpha v} \cap \Sigma_{\alpha} \right) \setminus r_{U_1}$ there passes exactly one of the lines in $\left\{r_{P} \mid P \in t_1, P \ne U_1\right\}$. Since the $q$ points $P \in t_1$ such that $r_{P} \cap \pi_{\alpha v}$ belong to $(\ell \setminus r_{U_1}) \cap \Sigma_{\alpha}$, together with $U_1$, form a Baer subline of $t_1$, it follows that $\cD_{\eta} \cap \cS_\ell$ is a regulus through $r_{U_1} \cap \Sigma_{\eta}$. Similarly if $\ell^{\tau_{\eta}} \in \cL$.

Viceversa, let  $\cS_{\ell}$ be a Desarguesian line-spread of $\Sigma_{\eta}$ such that $\cD_{\eta} \cap \cS_{\ell}$ is a regulus, denoted by $\cR_{\ell}$, through the line $r_{U_1} \cap \Sigma_{\eta}$. Since $r_{U_1} \in \cR_{\ell}$, by Lemma~\ref{prop0}, we infer that $\ell$ meets $r_{U_1}$ at exactly one point and $\ell \cap \Sigma_{\alpha}$ is a Baer subline, for some $\alpha \in \Lambda$, not belonging to $\cD_{\alpha}$. Finally, we see that at least one among $\ell$ and $\ell^{\tau_\eta}$ belongs to $\cL$. Indeed, let $\alpha \in \Lambda$ with $\alpha^{q+1} \ne -1$. If $\alpha \in \cI$, then $\ell \in \cL$, otherwise $\ell^{\tau_{\eta}} \in \cL$, by \eqref{indici}. In the case when $\alpha^{q+1} = -1$, then $\ell, \ell^{\tau_\eta} \in \cL$, since $\alpha \in \cI$ and $\Sigma_\alpha^{\tau_{\eta}} = \Sigma_{\alpha}$. 
\end{proof}

\begin{remark}\label{rem1}
For a line $\ell \in \cL$, denote by $\cR_\ell$ the regulus $\cD_{\eta} \cap \cS_\ell$ and by $\cR_\ell^{o}$ its opposite regulus. Clearly, $\left(\cS_\ell \setminus \cR_{\ell}\right) \cup \cR_{\ell}^{o}$ is a line-spread of $\Sigma_{\eta}$ known as {\em subregular spread of index one} or {\em Hall spread}, see \cite[pag. 54]{hirschfeld1985finite}, \cite[Section 8]{BruckBose}.
\end{remark}

Next, we focus on some particular lines of $\cL$, namely the $q$ lines meeting $r_{U_1}$ in $P_{\alpha}$ and contained in $\pi_{\alpha}$, where $\alpha \in \cI$. In order to do so, fix $\alpha \in \cI$, $\tilde{x} \in \F_{q^2} \setminus \F_q$ and set 
\begin{align*}
& l_\lambda = \langle P_{\alpha} = (1,0, \alpha, 0), (\lambda \tilde{x}, 1, \alpha \lambda \tilde{x}^q, \alpha) \rangle_{q^2}, & \lambda \in \F_q.
\end{align*}
Note that $l_\lambda \in \cL$, indeed, 
\begin{align*}
& l_\lambda \cap r_{U_1} = P_{\alpha}, \\
& l_\lambda \subset \pi_{\alpha}, \\ 
& l_\lambda \cap \Sigma_{\alpha} \mbox{ is a Baer subline, } & \lambda \in \F_q.
\end{align*}
 Let $\cS_{l_\lambda}$ be the Desarguesian line-spread of $\Sigma_{\eta}$ with transversal lines $l_\lambda$ and $l_\lambda^{\tau_{\eta}}$, where
\begin{align*}
& l_\lambda^{\tau_{\eta}} = \langle (\alpha^q, 0, 1, 0), (\alpha^q \lambda \tilde{x}, \alpha^q, \lambda \tilde{x}^q, 1) \rangle_{q^2}, & \lambda \in \F_q. 
\end{align*}
%\textcolor{red}{Da qui riformulei come in blue}
Now, consider the projectivities $\varphi$ and $\xi_{\lambda}$ represented by the matrices
\begin{align*}
& \begin{pmatrix}
1 & 0 & \alpha^q & 0 \\
0 & 1 & 0 & \alpha^q \\
\alpha & 0 & 1 & 0 \\
0 & \alpha & 0 & 1
\end{pmatrix}, \\
& \begin{pmatrix}
1 & \lambda \tilde{x} & 0 & 0 \\
0 & 1 & 0 & 0 \\
0 & 0 & 1 & \lambda \tilde{x}^q \\
0 & 0 & 0 & 1
\end{pmatrix}, & \lambda \in \F_q,
\end{align*}
respectively, and set 
\begin{align*}
& \phi_{\lambda} = \varphi \circ \xi_{\lambda}, & \lambda \in \F_q.
\end{align*}
It is easily seen that $\varphi$ fixes $\Sigma_{\eta}$ and $r_{U_1}$. Moreover, 
\begin{align*}
 l_0 & = t_1^\varphi, \\
 l_0^{\tau_{\eta}} & = t_2^\varphi. 
\end{align*}
On the other hand, $\xi_{\lambda}$ fixes $\overline{\cD}_{\eta}$. In particular, it stabilizes $t_1$, $t_2$, and each of the subgeometries $\Sigma_{\alpha}$, $\alpha \in \Lambda$. More precisely, $\xi_{\lambda}$ fixes the line $r_{U_1}$ pointwise and leaves invariant the plane $\pi_{\alpha v}$, for every $v \in \cU$. Furthermore, $\xi_{\lambda}$ maps $\cS_{l_0}$ to the Desarguesian line-spread $\cS_{l_\lambda}$ of $\Sigma_{\eta}$, since 
\begin{align*}
 l_0^{\xi_{\lambda}} & = l_{\lambda}, \\
 \left(l_0^{\tau_{\eta}}\right)^{\xi_{\lambda}} & = l_\lambda^{\tau_{\eta}}. 
\end{align*}
It follows that 
%Now, consider the projectivity $\phi_{\alpha_i,\lambda}$ represented by the matrix 
%\begin{align*}
%& \begin{pmatrix}
%1 & \lambda \tilde{x} & \alpha_i^q & \alpha_i^q \lambda \tilde{x} \\
%0 & 1 & 0 & \alpha_i^q \\
%\alpha_i & \alpha_i\lambda \tilde{x}^q & 1 & \lambda \tilde{x}^q \\
%0 & \alpha_i & 0 & 1
%\end{pmatrix}, &  \lambda \in \F_q.
%\end{align*}
%It is easily seen that $\phi_{\alpha_i,\lambda}$ fixes $\Sigma_{\eta}$. Moreover, 
\begin{align*}
 & t_1^{\phi_{\lambda}} = t_1^{\varphi \circ \xi_{\lambda}} = \left(t_1^\varphi\right)^{\xi_{\lambda}} = l_0^{\xi_{\lambda}} = l_\lambda, \\
& t_2^{\phi_{\lambda}}  = t_2^{\varphi \circ \xi_{\lambda}} = \left(t_2^\varphi\right)^{\xi_{\lambda}} = \left(l_0^{\tau_{\eta}}\right)^{\xi_{\lambda}} = l_\lambda^{\tau_{\eta}}. 
\end{align*}
Hence the image of $\cD_{\eta}$ under $\phi_{\lambda}$ is the Desarguesian line-spread $\cS_{l_{\lambda}}$ of $\Sigma_{\eta}$. Note that 
\begin{equation} \label{eq:imageSlambda}
\overline{\cS}_{l_{\lambda}} = \bigcup_{\rho \in \Lambda} \Sigma_{\rho}^{\phi_{\lambda}} \cup l_{\lambda} \cup l_{\lambda}^{\tau_{\eta}},
\end{equation}
where $\Sigma_{\rho}^{\phi_{\lambda}}$, $\rho \in \Lambda$, are pairwise disjoint Baer subgeometries, each isomorphic to $\PG(3, q)$. 

\begin{prop} \label{prop:interesctionSllambdapi}
Let $\beta \in \cI$ and $v \in \cU$, the intersection $\overline{\cS}_{l_\lambda}  \cap \pi_{\beta v}$ consists of $2q^2+1$ points. In particular, 
\begin{align*}
\overline{\cS}_{l_\lambda}  \cap \pi_{\beta v} =
\begin{cases}
r_{U_1} \cup l_\lambda, & \mbox{ if } (\beta, v) = (\alpha, 1), \\
r_{U_1} \cup l_\lambda^{\tau_{\eta}}, & \mbox{ if } q \mbox{ is odd, } \alpha^{q+1} = -1 \mbox{ and } (\beta, v) = (\alpha, -1), \\
r_{U_1} \cup \sigma_{\beta v, \lambda}, \mbox{ with } \sigma_{\beta v, \lambda} \mbox{ a Baer subplane, }  & \mbox{ otherwise. }%\mbox{ if } (\beta, v) \ne (\alpha, 1).
\end{cases}
\end{align*}
\end{prop}
\begin{proof}
First, note that $\left(\overline{\mathcal{D}}_\eta\right)^{\phi_{\lambda}} = \overline{\cS}_{l_\lambda}$ and that $\phi_{\lambda}$ fixes $r_{U_1}$, and hence $r_{U_1} \subseteq \overline{\cS}_{l_\lambda} \cap \pi_{\beta v}$. By using {\em ii)} of Lemma \ref{lemma2}, we have that $\lvert \overline{\cS}_{l_\lambda} \cap \pi_{\beta v} \rvert = 2q^2 + 1$. Moreover,
if $(\beta, v) = (\alpha, 1)$, we get that the point $(\lambda \tilde{x}, 1, \alpha \lambda \tilde{x}^q, \alpha)$ belongs to $\pi_{\beta v}$, proving that $l_\lambda$ is contained in  $\pi_{\beta v}$ and hence $\overline{\cS}_{l_\lambda} \cap \pi_{\beta v} = r_{U_1} \cup l_\lambda$. Similarly, if $q$ is odd, $\alpha^{q+1} = -1$ and $(\beta, v) = (\alpha, -1)$, we have that $\overline{\cS}_{l_\lambda} \cap \pi_{\beta v} = r_{U_1} \cup l_\lambda^{\tau_\eta}$.  Otherwise, %if $(\beta, v) \neq (\alpha, 1)$, 
it is easy to check that neither $l_\lambda$ nor $l_\lambda^{\tau_{\eta}}$ is contained in $\pi_{\beta v}$, which, again by using {\em ii)} of Lemma~\ref{lemma2}, proves that
$\overline{\cS}_{l_\lambda} \cap \pi_{\beta v} = r_{U_1} \cup \sigma_{\beta v, \lambda}$,
with $\sigma_{\beta v, \lambda}$ a Baer subplane.
\end{proof}

\begin{prop}\label{prop1}
Let $\lambda \in \F_q$, $\beta \in \cI$ and $v \in \cU$, with $v \ne 1$. The Baer subplanes $\sigma_{\beta v, \lambda}$ and $\Sigma_{\beta} \cap \pi_{\beta v}$ share a configuration consisting of the $q+2$ points given by $P_{\frac{\beta^q \alpha}{\alpha^q} v^q}$ and a Baer subline $s_{\beta v, \lambda}$. 
\end{prop}
\begin{proof}
First consider $\lambda=0$. Hence $\phi_0 = \varphi$. Note that, by \eqref{eq:imageSlambda}, we have 
\begin{align*}
\overline{\cS}_{l_0} = \bigcup_{\rho \in \Lambda} \Sigma_{\rho}^{\phi_{0}} \cup l_0 \cup l_0^{\tau_{\eta}},
\end{align*}
where $\Sigma_{\rho}^{\phi_{0}}$, $\rho \in \Lambda$, are Baer subgeometries, each isomorphic to $\PG(3, q)$, whose points have homogeneous projective coordinates 
\begin{align*}
& \left(x + \alpha^q \rho x^q, y + \alpha^q \rho y^q, \alpha x + \rho x^q, \alpha y + \rho y^q \right), & x, y \in \F_{q^2} \mbox{ and } (x, y) \ne (0, 0). 
\end{align*}
%\textcolor{red}{Da qui è uguale a quello che hai scritto tu.}
Assume that $\beta \in \cI$ and $v \in \cU$, with $v \ne 1$. Exactly one of these $q-1$ subgeometries meets $\pi_{\beta v}$ in the Baer subplane $\sigma_{\beta v, 0}$. Some calculations show that $\Sigma_{\rho}^{\phi_0} \cap \pi_{\beta v}$ is a Baer subplane if and only if $\rho = \frac{\beta v -\alpha}{\beta \alpha^q v - 1}$. Moreover, 
\begin{align*}
& \sigma_{\beta v, 0} = \left\{ R_{x,y} \mid x, y \in \F_{q^2}, y+ y^q = 0, (x,y) \ne (0,0) \right\}, \mbox{ where } \\
& R_{x,y} = \left(x + \alpha^q \frac{\beta v -\alpha}{\beta \alpha^q v - 1} x^q, y + \alpha^q \frac{\beta v -\alpha}{\beta \alpha^q v - 1} y^q, \alpha x + \frac{\beta v -\alpha}{\beta \alpha^q v - 1} x^q, \alpha y + \frac{\beta v -\alpha}{\beta \alpha^q v - 1} y^q \right),
%\left( \left(1-\alpha_2^{q+1} v\right) x + \alpha_2^{q+1}\left(1-v\right) x^q, y \left(1-\alpha_2^{q+1}\right), \alpha_2 \left(1-\alpha_2^{q+1} v\right) x + \alpha_2\left(1-v\right) x^q, \alpha_2 y \left(1-\alpha_2^{q+1}\right) v \right).
\end{align*}
A point $R_{x,y}$ of $\sigma_{\beta v, 0}$ belongs to the Baer subplane $\Sigma_{\beta} \cap \pi_{\beta v}$ if and only if $R_{x,y}^{\tau_{\beta}} = R_{x,y}$, which in turn implies 
\begin{align*}
%\left( \left(1-\alpha_2^{q+1} v\right) x + \alpha_2^{q+1}(1-v) x^q \right)^{q+1} - \left( \left(1-\alpha_2^{q+1} v\right) x + (1-v) x^q \right)^{q+1} & = 0, 
\beta^{q+1} \left( \left(\beta \alpha^q v - 1\right) x + \alpha^q\left(\beta v - \alpha\right) x^q \right)^{q+1} - \left( \alpha\left(\beta \alpha^q v - 1\right) x + \left(\beta v -\alpha\right) x^q \right)^{q+1} & = 0,
\end{align*}
that is 
\begin{align*}
%\left( x+x^q \right) \left( \left(1-\alpha_2^{q+1} v\right) x - \left(v - \alpha_2^{q+1}\right) x^q  \right) & = 0.
\left( x+x^q \right) \left( \alpha \left(\beta \alpha^q v - 1\right) \left( \beta^q v^q - \alpha^q \right) x + \alpha^q \left(\beta^q \alpha v^q - 1\right) \left(\beta v - \alpha\right) x^q  \right) & = 0.
\end{align*}
If $x + x^q = 0$, we obtain the $q+1$ points of $\sigma_{\beta v, 0} \cap \Sigma_{\beta} \cap \pi_{\beta v}$ given by the Baer subline
\begin{align*}
s_{\beta v, 0} = \left\{ (x,y, \beta v x, \beta v y) = x P_{\beta v} + y (0,1,0,\beta v), x, y \in \F_{q^2}, x+x^q = y + y^q = 0, (x, y) \ne (0,0) \right\}.
\end{align*} 
If $x^q = - \frac{\alpha \left(\beta \alpha^q v - 1\right) \left( \beta^q v^q - \alpha^q \right) x}{\alpha^q \left(\beta^q \alpha v^q - 1\right) \left(\beta v - \alpha\right)}
$ and $R_{x,y} \in \sigma_{\beta v, 0} \cap \Sigma_{\beta} \cap \pi_{\beta v}$, then $y = 0$. Hence, in this case, the unique point of $\sigma_{\beta v, 0} \cap \Sigma_{\beta} \cap \pi_{\beta v}$ is
\begin{align*}
P_{\frac{\beta^q \alpha}{\alpha^q} v^q} = \left(1,0,\frac{\beta^q \alpha}{\alpha^q} v^q, 0 \right).
\end{align*}
We have just seen that the two Baer subplanes $\sigma_{\beta v, 0}$ and $\Sigma_{\beta} \cap \pi_{\beta v}$ share precisely the configuration formed by the $q+2$ points given by $P_{\frac{\beta^q \alpha}{\alpha^q} v^q}$ and the Baer subline $s_{\beta v, 0}$.

%Let $\xi_{\lambda}$ denote the projectivity represented by the matrix
%\begin{align*}
%& \begin{pmatrix}
%1 & \lambda \tilde{x} & 0 & 0 \\
%0 & 1 & 0 & 0 \\
%0 & 0 & 1 & \lambda \tilde{x}^q \\
%0 & 0 & 0 & 1
%\end{pmatrix}, & \lambda \in \F_q. 
%\end{align*}
%Then $\xi_{\lambda}$ fixes $\overline{\cD}_1$. In particular, it stabilizes $t_1$, $t_2$, and 
Recall that $\xi_{\lambda}$ fixes each of the subgeometries $\Sigma_{\rho}$, $\rho \in \Lambda$. The line $r_{U_1}$ is fixed pointwise by $\xi_{\lambda}$ and the plane $\pi_{\beta v}$ is left invariant by $\xi_{\lambda}$. Furthermore, $\xi_{\lambda}$ maps $\cS_{l_0}$ to the Desarguesian line-spread $\cS_{l_\lambda}$ of $\Sigma_{\eta}$, since 
\begin{align*}
 l_0^{\xi_{\lambda}} & = l_{\lambda}, \\
 \left(l_0^{\tau_{\eta}}\right)^{\xi_{\lambda}} & = l_\lambda^{\tau_{\eta}}. 
\end{align*}
It follows that 
\begin{align*}
  \left( \Sigma_{\beta} \cap \pi_{\beta v} \right)^{\xi_{\lambda}} =  \Sigma_{\beta}^{\xi_{\lambda}} \cap \pi_{\beta v}^{\xi_{\lambda}} & = \Sigma_{\beta} \cap \pi_{\beta v}, \\
  \overline{\cS}_{l_0}^{\xi_{\lambda}} = \overline{\cS}_{l_\lambda} \implies \sigma_{\beta v, 0}^{\xi_{\lambda}} & = \sigma_{\beta v, \lambda}.
\end{align*}
Therefore we may conclude that the intersection of the two Baer subplanes $\sigma_{\beta v, \lambda}$ and $\Sigma_{\beta} \cap \pi_{\beta v}$ consists of the $q+2$ points given by $P_{\frac{\beta^q \alpha}{\alpha^q} v^q}$ and the Baer subline $s_{\beta v, \lambda}$, where $s_{\beta v, \lambda} = s_{\beta v, 0}^{\xi_{\lambda}}$.
\end{proof}

\begin{prop}\label{prop2}
Let $\lambda \in \F_q$, $\beta \in \cI$ and $v \in \cU$. If $\ell \subset \pi_{\beta v}$ is a line of $\cL \setminus \{l_\lambda\}$ intersecting non trivially $\overline{\cS}_{l_\lambda} \setminus \overline{\cR}_{l_\lambda}$, then $\ell$ meets $r_{U_1}$ in the point $P_{\frac{\beta^q \alpha}{\alpha^q} v^q}$. 
\end{prop}
\begin{proof}
Since $\ell \cap \overline{\cD}_{\eta} \subset \Sigma_{\beta}$ and $\overline{\cS}_{l_\lambda} \cap \overline{\cD}_{\eta} = \overline{\cR}_{l_\lambda}$, it follows that $\ell \cap \left( \overline{\cS}_{l_\lambda} \setminus \overline{\cR}_{l_\lambda} \right) = \ell \cap \left( \overline{\cS}_{l_\lambda} \setminus \Sigma_{\beta} \right)$. Moreover, $\ell \subset \pi_{\beta v}$, for some $v \in \cU$, and hence 
\begin{align*}
\ell \cap \left( \overline{\cS}_{l_\lambda} \setminus \overline{\cR}_{l_\lambda} \right) = \ell \cap \left( \left( \overline{\cS}_{l_\lambda}  \cap \pi_{\beta v} \right) \setminus \left( \Sigma_{\beta} \cap \pi_{\beta v} \right) \right).
\end{align*} 
Recall that 
\begin{align*}
& \overline{\cS}_{l_\lambda}  \cap \pi_{\beta v} = r_{U_1} \cup l_\lambda, & & \mbox{ if } (\beta, v) = (\alpha, 1), \\
& \overline{\cS}_{l_\lambda}  \cap \pi_{\beta v} = r_{U_1} \cup l_\lambda^{\tau_{\eta}}, & & \mbox{ if } q \mbox{ is odd, } \alpha^{q+1} = -1 \mbox{ and }(\beta,v)=(\alpha,-1), \\
& \overline{\cS}_{l_\lambda}  \cap \pi_{\beta v} = r_{U_1} \cup \sigma_{\beta v, \lambda}, \mbox{ with } \sigma_{\beta v, \lambda} \mbox{ a Baer subplane, } & & \mbox{ otherwise. } %\mbox{ if } (\beta, v) \ne (\alpha, 1).
\end{align*}
Moreover, $\ell \cap r_{U_1} \in \Sigma_{\beta}$ and $\Sigma_{\beta} \cap \pi_{\beta v}$ is a Baer subplane. If $(\beta, v) = (\alpha, 1)$, then $\ell \cap \Sigma_{\beta}$, $l_\lambda \cap \Sigma_{\beta}$ are distinct Baer sublines of the Baer subplane $\Sigma_{\alpha} \cap \pi_{\alpha}$, therefore $\ell \cap l_\lambda \in \Sigma_{\alpha} \cap \pi_{\alpha}$. Similarly, if $q$ is odd, $\alpha^{q+1} = -1$ and $(\beta, v) = (\alpha, -1)$, then $\ell \cap l_\lambda^{\tau_{\eta}} \in \Sigma_{\alpha} \cap \pi_{-\alpha}$, and hence 
\begin{align*}
& \ell \cap \left( \overline{\cS}_{l_\lambda} \setminus \overline{\cR}_{l_\lambda} \right) = \emptyset, & \mbox{ if } \ell \subset \pi_{\alpha v}, \mbox{ where } v \in \{\pm 1\}.
\end{align*} 
In order to prove the statement we will show that if $\overline{\cS}_{l_\lambda}  \cap \pi_{\beta v} = r_{U_1} \cup \sigma_{\beta v, \lambda}$, where $\beta \in \cI$, $v \in \cU$, and $\ell$ is a line of $\pi_{\beta v}$ such that $\ell \cap \Sigma_{\beta} \cap \pi_{\beta v}$ is a Baer subline and $\ell$ intersects $\sigma_{\beta v, \lambda} \setminus \Sigma_{\beta}$ in at least one point, then necessarily $P_{\frac{\beta^q \alpha}{\alpha^q} v^q} \in \ell$. By Proposition~\ref{prop1}, the two Baer subplanes $\sigma_{\beta v, \lambda}$ and $\Sigma_{\beta} \cap \pi_{\beta v}$ share precisely a configuration consisting of the $q+2$ points given by $P_{\frac{\beta^q \alpha}{\alpha^q} v^q}$ and the Baer subline $s_{\beta v, \lambda}$. By \cite{bose1980intersection}, there are precisely $q+2$ lines of $\pi_{\beta v}$ intersecting both $\sigma_{\beta v, \lambda}$ and $\Sigma_{\beta} \cap \pi_{\beta v}$ in a Baer subline, they are obtained by joining $P_{\frac{\beta^q \alpha}{\alpha^q} v^q}$ with a point of $s_{\beta v, \lambda}$, together with the extension of $s_{\beta v, \lambda}$ over $\F_{q^2}$. Therefore if $\ell$ is a line of $\pi_{\beta v}$ such that $\ell \cap \Sigma_{\beta} \cap \pi_{\beta v}$ is a Baer subline and $\ell$ intersects $\sigma_{\beta v, \lambda} \setminus \Sigma_{\beta}$ in at least one point, then necessarily $P_{\frac{\beta^q \alpha}{\alpha^q} v^q} \in \ell$.
\end{proof}

\section{Parallelisms of \texorpdfstring{$\PG(3, q)$}{PG(3, q)}}

With the same notation used before, denote by 
\begin{align*}
\left\{P_{\alpha u} \colon u \in \cU\right\}
\end{align*}
the $q+1$ points of the Baer subline $r_{U_1} \cap \Sigma_{\alpha}$, and by 
\begin{align*}
\left\{\pi_{\alpha v} \colon v \in \cU\right\}
\end{align*}
the $q+1$ planes of $\PG(3, q^2)$ through $r_{U_1}$ intersecting $\Sigma_{\alpha}$ in a Baer subplane, where
\begin{align*}
&  \alpha \in \cI, \\
& P_{\alpha u} = (1,0,\alpha u, 0), & u \in \cU, \\
& \pi_{\alpha v}: X_4 = \alpha v X_2, & v \in \cU.
\end{align*}

Recall that $p\left(P_{\alpha u}, \pi_{\alpha v}\right)$ denotes the Baer subpencil formed by the $q+1$ lines through $P_{\alpha u}$ contained in $\pi_{\alpha v}$ and meeting $\pi_{\alpha v} \cap \Sigma_{\alpha}$ in a Baer subline. %Note that $r_{U_1} \in p(P_{\alpha_k u_i}, \pi_{\alpha_k v_i})$, $\forall \alpha_k \in \cI$, $\forall u_i, v_i \in \cU$. 

\begin{prop}\label{mainp1}
Let $\ell_i$ and $\ell_{j}$ be distinct lines of $\cL$, such that $\ell_i \in p\left( P_{\alpha u_i}, \pi_{\alpha v_i} \right)$, $\ell_{j} \in p\left( P_{\beta u_{j}}, \pi_{\beta v_{j}} \right)$. 
Then $\cR_{\ell_{i}} \ne \cR_{\ell_{j}}$ if and only if either 
\begin{align*}
  & (\alpha, u_i, v_i) = (\beta, u_{j}, v_{j})  
\end{align*}
or
\begin{align*}
& u_{i} v_{j} - u_{j} v_{i} \ne 0. 
\end{align*}
\end{prop}
\begin{proof}
Consider the six-dimensional $\F_q$-vector spaces
\begin{align*}
& V_{\rho} =  \left\{(a,b, \alpha_i c, - \rho c^q, d, \rho^2 a^q) \mid a,b,c,d \in \F_{q^2}, \rho^q b + \rho b^q = \rho^q d + \rho d^q = 0 \right\}, & \rho \in \Lambda,  
\end{align*}
and the Klein quadric $\cQ^+(5, q^2)$ of $\PG(5, q^2)$:
\begin{align*}
& \cQ^+(5, q^2): X_1 X_6 - X_2 X_5 + X_3 X_4 = 0.  
\end{align*}
Set 
\begin{align*}
& \tilde{\Sigma}_{\rho} = \PG(V_{\rho}) \simeq \PG(5, q), & \rho \in \Lambda.
\end{align*}
Some calculations show that the lines of $\PG(3, q^2)$ intersecting $\Sigma_{\rho}$ in a Baer subline are embedded via Pl\"ucker coordinates to the set of points $\cQ_{\rho}$, where
\begin{align*}
&\cQ_{\rho} = \tilde{\Sigma}_{\rho} \cap \cQ^+(5, q^2) \simeq \cQ^+(5, q), & \rho \in \Lambda. 
\end{align*}
By Lemma~\ref{lemma1}, {\em ii)}, we have that 
\begin{align*}
& \cQ_{\rho} \cap \cQ_{\rho'} = \cQ^-(3, q) = \left\{T_a = \left( 0, a^{q+1}, a, a^q, 1, 0 \right) \mid a \in \F_{q^2} \right\} \cup \left\{(0,1,0,0,0,0)\right\}, & \rho \ne \rho'.
\end{align*}
Note that, for some fixed $\lambda_1, \lambda_2 \in \F_{q^2} \setminus \{0\}$, $\lambda_1^{q-1} \ne \lambda_2^{q-1}$,  
\begin{align*}
& \rho \left( \lambda_1 \lambda_2^q - \lambda_1^q \lambda_2 \right) \left( 0,  x^{q+1}, x y^q, x^q y, y^{q+1}, 0 \right) \in V_{\rho}, & x,y \in \F_{q^2}, y \ne 0,
\end{align*}
is a representative of the point $T_a$, where $a = x/y$.
Moreover, the line $r_{U_1}$ is mapped by the Grassmann embedding to the point $T_\infty = (0,1,0,0,0,0)$. 

Let $\ell_i \in p\left( P_{\alpha u_i}, \pi_{\alpha v_i} \right)$, $\ell_j \in p\left( P_{\beta u_{j}}, \pi_{\beta v_{j}} \right)$, $\ell_i \ne \ell_j$, $\ell_i \ne r_{U_1}$, $\ell_j \ne r_{U_1}$. Then there exists $u' \in \cU \setminus \{u_i, u_{j}\}$ such that $\ell_i$ is the line joining $P_{\alpha u_i}$ and a point of $\langle P_{\alpha u'}, (0,1,0, \alpha v_i) \rangle_{q^2} \setminus \{P_{\alpha u'}\}$. Similarly, $\ell_j$ is the line joining $P_{\beta u_j}$ and a point of $\langle P_{\beta u'}, (0,1,0, \beta v_j) \rangle_{q^2} \setminus \{P_{\beta u'}\}$. Therefore, there exist $\lambda_1, \lambda_2 \in \F_q$, such that 
\begin{align*}
& \ell_i = \left\langle \left(x_i, 0, \alpha x_i^q, 0 \right), \left(\lambda_1 x', y_i, \lambda_1 \alpha x'^q, \alpha y_i^q \right) \right\rangle_{q^2}, \\
& \ell_j = \left\langle \left(x_{j}, 0, \beta x_{j}^q, 0 \right), \left(\lambda_2 x', y_{j}, \lambda_2 \beta x'^q, \beta y_{j}^q \right) \right\rangle_{q^2}, 
\end{align*}
where $x_i, y_i, x_{j}, y_{j}, x' \in \F_{q^2}$ are such that $x_i^{q-1} = u_i$, $y_i^{q-1} = v_i$, $x_{j}^{q-1} = u_{j}$, $y_{j}^{q-1} = v_{j}$, $x'^{q-1} = u'$. The Pl\"ucker coordinates of the lines $\ell_i$, $\ell_j$ are $L_i$, $L_{j}$, where
\begin{align*}
& L_i = \left( x_i y_i, \lambda_1 \alpha \left(x_1 x'^q - x_1^q x'\right), \alpha x_i y_i^q, - \alpha x_i^q y_i, 0, \alpha^2 x_i^q y_i^q \right), \\
& L_{j} = \left( x_{j} y_{j}, \lambda_2 \beta \left(x_{j} x'^q - x_{j}^q x'\right), \beta x_{j} y_{j}^q, - \beta x_{j}^q y_{j}, 0, \beta^2 x_{j}^q y_{j}^q \right).
\end{align*}
Let $\perp_{\rho}$ denote the polarity of $\Sigma_{\rho}$ associated with $\cQ_{\rho}$ and observe that $L_i^\perp \cap \cQ^-(3, q)$ is a conic of $\cQ^-(3, q)$ through the point $T_\infty$. Such a conic is the image of the regulus $\cR_{\ell_i}$ under the Grassmann embedding. Therefore  
\begin{align*}
& \cR_{\ell_i} = \cR_{\ell_j} \iff L_i^{\perp_{\alpha}} \cap \cQ^-(3, q) = L_{j}^{\perp_{\beta}} \cap \cQ^-(3, q).
\end{align*}
Furthermore,
\begin{align*}
& T_\infty \in L_i^{\perp_{\alpha}} \cap \cQ^-(3, q) \cap L_{j}^{\perp_{\beta}} \cap \cQ^-(3, q), \\
& |L_i^{\perp_{\alpha}} \cap \cQ^-(3, q) \cap L_{j}^{\perp_{\beta}} \cap \cQ^-(3, q)| \ge 3 \iff L_i^{\perp_{\alpha}} \cap \cQ^-(3, q) = L_{j}^{\perp_{\beta}} \cap \cQ^-(3, q).
\end{align*}
For a point $T_a = (0, a^{q+1}, a, a^q, 1, 0) \in \cQ^-(3, q)$, we have that 
\begin{align*}
T_a \in L_i^{\perp_{\alpha}} & \iff - \alpha x_i^q y_i a + \alpha x_i y_i^q a^q - \alpha \lambda_1 \left(x_i x'^q - x_i^q x'\right) = 0 & \\
& \iff a^q - \frac{x_i^q y_i}{x_i y_i^q} a - \lambda_1 \frac{x_i x'^q - x_i^q x'}{x_i y_i^q} = 0.
\end{align*}
Similarly, 
\begin{align*}
T_a \in L_{j}^{\perp_{\beta}} \iff a^q - \frac{x_{j}^q y_{j}}{x_{j} y_{j}^q} a - \lambda_2 \frac{x_{j} x'^q - x_{j}^q x'}{x_{j} y_{j}^q} = 0. && 
\end{align*}
Suppose that $(\alpha, u_i, v_i) = (\beta, u_{j}, v_{j})$. Then we may assume $x_i = x_{j}$ and $y_i = y_{j}$. We deduce that necessarily $\lambda_1 \ne \lambda_2$, otherwise $\ell_i = \ell_j$. If $T_a \in L_i^{\perp_{\alpha}} \cap L_{j}^{\perp_{\beta}}$, then $\lambda_1 = \lambda_2$, a contradiction. Therefore  
$\cR_{\ell_i} \ne \cR_{\ell_j}$.

Suppose that $(\alpha, u_i, v_i) \ne (\beta, u_{j}, v_{j})$. By \cite[Corollary 1.24]{hirschfeld1998projective}, since 
\begin{align*}
\left(\frac{\lambda_1 \left(x_i x' - x_i^q x'\right)}{x_i y_i^q}\right)^q & = \lambda_1 \frac{x_i^q x' - x_i x'^q}{x_i^q y_i} \\
& - \lambda_1 \frac{x_i x'^q - x_i^q x'}{x_i y_i^q} \cdot \frac{x_i y_i^q}{x_i^q y_i},
\end{align*}
the equation
\begin{align}
& X^q - \frac{x_i^q y_i}{x_i y_i^q} X - \lambda_1 \frac{x_i x'^q - x_i^q x'}{x_i y_i^q} = 0, \label{eq1}
\end{align}
in the indeterminate $X$ has $q$ solution, as expected. In particular, $T_{a_1}, T_{a_2}$ are distinct points of $L_i^{\perp_{\alpha}}$ if and only if $a_1, a_2 \in \F_{q^2}$ are distinct solutions of \eqref{eq1}, that is 
\begin{align*}
& a_2 = \frac{\mu}{x_i^q y_i} + a_1, & \mbox{ for some } \mu \in \F_q \setminus \{0\}.
\end{align*}
Analogously, if $T_{a_1}, T_{a_2}$ are distinct points of $L_{j}^{\perp_{\beta}}$, then 
\begin{align*}
& a_2 = \frac{\mu'}{x_{j}^q y_{j}} + a_1, & \mbox{ for some } \mu' \in \F_q \setminus \{0\}.
\end{align*}
Assume by contradiction that $\cR_{\ell_i} = \cR_{\ell_j}$, then there are three points $T_{\infty}, T_{a_1}, T_{a_2} \in L_i^{\perp_{\alpha}} \cap \cQ^-(3, q) \cap L_{j}^{\perp_{\beta}} \cap \cQ^-(3, q)$. By the previous arguments, the following holds
\begin{align*}
& \frac{\mu}{x_i^q y_i} = \frac{\mu'}{x_{j}^q y_{j}}, & \mbox{ for some } \mu, \mu' \in \F_q \setminus \{0\}.
\end{align*}
Hence 
\begin{align*}
\left(\frac{\mu}{x_i^q y_i}\right)^{q-1} = \left(\frac{\mu'}{x_{j}^q y_{j}}\right)^{q-1} & \iff \frac{x_i^q y_i}{x_i y_i^q} = \frac{x_{j}^q y_{j}}{x_{j} y_{j}^q} \\
& \iff u_i v_{j} =  u_{j} v_i,
\end{align*}
a contradiction.
\end{proof}

\begin{prop}\label{mainp2}
Let $\ell_{i}$ and $\ell_j$ be distinct lines of $\cL$, such that $\ell_{i} \in p\left( P_{\alpha u_i}, \pi_{\alpha v_i} \right)$, $\ell_j \in p\left( P_{\beta u_{j}}, \pi_{\beta v_{j}} \right)$. Then $\ell_j \cap \left( \overline{\cS}_{\ell_i} \setminus \overline{\cR}_{\ell_i} \right) = \emptyset$ if and only if either 
\begin{align*}
  & (\alpha, u_i, v_i) = (\beta, u_j, v_j)  
\end{align*}
or
\begin{align*}
& \alpha u_i \left(\beta v_j\right)^q - \left(\alpha v_i\right)^q \beta u_j \ne 0. 
\end{align*}
\end{prop}
\begin{proof}
Consider the projectivity $\gamma$ represented by the matrix
\begin{align*}
\begin{pmatrix}
1 & 0 & 0 & 0 \\
0 & b & 0 & 0 \\
0 & 0 & u_i^{-1} & 0 \\
0 & 0 & 0 & b v_i^{-1}
\end{pmatrix}, 
\end{align*}
for some $b \in \F_{q^2}$, such that $b^{q-1} = \frac{u_i}{v_i}$. Then $\gamma$ fixes $\overline{\cD}_{\eta}$. In particular, $\gamma$ stabilizes $t_1$, $t_2$ and $r_{U_1}$. For any $\rho \in \Lambda$, a point $\left(x,y,\rho x^q, \rho y^q\right)$ of $\Sigma_{\rho}$ is sent by $\gamma$ to $\left(x, by, \rho u_i^{-1} x^q, \rho b v_i^{-1} y^q\right)$ that belongs to $\Sigma_{\rho}$. Indeed, let $c \in \F_{q^2}$ with $c^{q-1} = v_i^{-1}$, then $c^q = c v_i^{-1}$, $c^q b = b^q c u_i^{-1}$ and $\left(x, by, \rho u_i^{-1} x^q, \rho b v_i^{-1} y^q\right)$ represents the point
\begin{align*}
b^q c \left(x, by, \rho u_i^{-1} x^q, \rho b v_i^{-1} y^q\right) & = \left( b^qc x, b^{q+1} c y, \rho b^q c u_i^{-1} x^q, \rho b^{q+1} c v_i^{-1} y^q \right) \\
& = \left( b^q c x, b^{q+1} c y, \rho b c^q x^q, \rho b^{q+1} c^q y^q \right)
\end{align*}
Therefore, each of the subgeometries $\Sigma_{\rho}$, with $\rho \in \Lambda$, is left invariant by $\gamma$. Assume by contradiction that $\ell_j \in p\left( P_{\beta u_j}, \pi_{\beta v_j} \right)$ is such that $\ell_j \cap \left( \overline{\cS}_{\ell_{i}} \setminus \overline{\cR}_{\ell_{i}} \right) \ne \emptyset$ and observe that 
\begin{align*}
\ell_j \cap \left( \overline{\cS}_{\ell_i} \setminus \overline{\cR}_{\ell_i} \right) \ne \emptyset \iff \ell_j^\gamma \cap \left( \overline{\cS}_{\ell_i}^\gamma \setminus \overline{\cR}_{\ell_i}^\gamma \right) \ne \emptyset
\end{align*}
Since $\gamma$ fixes $\Sigma_{\eta}$ and $\cS_{\ell_{i}}$ is the Desarguesian line-spread of $\Sigma_{\eta}$ with transversal lines $\ell_{i}$, $\ell_i^{\tau_{\eta}}$, then $\cS_{\ell_i}^\gamma$ is the Desarguesian line-spread of $\Sigma_{\eta}$ with transversal lines $\ell_i^\gamma$, $\left(\ell_i^\gamma\right)^{\tau_{\eta}}$. Hence 
\begin{align*}
& \overline{\cS}_{\ell_i}^\gamma = \overline{\cS}_{\ell_i^\gamma}, \\
& \overline{\cR}_{\ell_i}^\gamma = \overline{\cR}_{\ell_i^\gamma}.
\end{align*}
The line $\ell_i$ meets $r_{U_1}$ in the point $P_{\alpha u_i} = (1,0,\alpha u_i,0)$ and lies in the plane $\pi_{\alpha v_i}: X_4 = \alpha v_i X_2$. Hence the line $\ell_i^\gamma$ meets $r_{U_1}$ in the point $P_{\alpha u_i}^\gamma = (1, 0, \alpha, 0) = P_{\alpha}$ and lies in the plane $\pi_{\alpha v_i}^\gamma = \pi_{\alpha}: X_4 = \alpha X_2$. Moreover, $\ell_i^\gamma \cap \Sigma_{\alpha} = \left(\ell_i \cap \Sigma_{\alpha}\right)^\gamma$ is a Baer subline of $\Sigma_{\alpha}$, since $\ell_i \cap \Sigma_{\alpha}$ is a Baer subline. Therefore, 
\begin{align*}
& \ell_i^\gamma = l_\lambda, & \mbox{ for some } \lambda \in \F_q, 
\end{align*}
being $\ell_i^\gamma$ a line of $\cL$ meeting $r_{U_1}$ in $P_{\alpha}$ and contained in $\pi_{\alpha}$. Similarly, $\ell_j^\gamma$ is a line of $\cL$ intersecting $r_{U_1}$ in $P_{\beta u_j}^\gamma = (1,0,\beta u_j u_i^{-1}, 0)$ and contained in $\pi_{\beta v_j}^\gamma: X_4 = \beta v_j v_i^{-1} X_2$. By Proposition~\ref{prop2}, 
\begin{align*}
\ell_j^\gamma \cap \left( \overline{\cS}_{\ell_{i}}^\gamma \setminus \overline{\cR}_{\ell_{i}}^\gamma \right) \ne \emptyset \implies P_{\beta u_j}^\gamma = \left(1,0,\beta u_j u_i^{-1}, 0\right) = P_{\frac{\beta^q \alpha}{\alpha^q} v_j^q v_i^{-q}} = \left(1,0,\frac{\beta^q \alpha}{\alpha^q} v_j^q v_i^{-q}, 0\right).
\end{align*}   
If $(\alpha, u_i, v_i) = (\beta, u_j, v_j)$, then $\ell_j^\gamma \cap \left( \overline{\cS}_{\ell_i}^\gamma \setminus \overline{\cR}_{\ell_i}^\gamma \right) = \emptyset$ by Proposition~\ref{prop2}, a contradiction.
Otherwise, 
\begin{align*}
\alpha u_i \left(\beta v_j\right)^q = \left(\alpha v_i\right)^q \beta u_j, 
%\frac{u_j}{u_i} = \frac{v_j^q}{v_i^q} = \frac{v_i}{v_j}.
\end{align*}
%Set $u = \frac{u_i}{u_j} \in \cU$. 
%If $u \ne 1$, then 
%\begin{align*}
%u_j = \frac{u_i}{u}, v_j = u v_i,
%\end{align*}
a contradiction.
\end{proof}

For a set
\begin{align*}
& \cP = \left\{\left(P_{\alpha_{1} u_1}, \pi_{\alpha_{1} v_1}\right), \dots, \left(P_{\alpha_{{q+1}} u_{q+1}}, \pi_{\alpha_{{q+1}} v_{q+1}}\right)\right\}, & \alpha_{i} \in \cI, u_i, v_i \in \cU, 
\end{align*}
define the following set of lines 
\begin{align*}
\cL_{\cP} = \left\{p\left(P_{\alpha_{i} u_i}, \pi_{\alpha_{i} v_i}\right) \setminus \{r_{U_1}\} \mid \left(P_{\alpha_{i} u_i}, \pi_{\alpha_{i} v_i}\right) \in \cP \right\} \subset \cL. 
\end{align*} 
Since $|\cP| = q+1$, $\cL_{\cP}$ contains $q^2+q$ lines of $\PG(3, q^2)$. Furthermore, if $\ell$ is a line of $\cL_{\cP}$, then 
\begin{itemize}
\item[{\em i)}] $\ell \cap r_{U_1} \in P_{\alpha_{i} u_i}$, for some $\alpha_{i} \in \cI$, $u_i \in \cU$, 
\item[{\em ii)}] $\ell \subset \pi_{\alpha_{i} v_i}$,  for some $\alpha_{i} \in \cI$, $v_i \in \cU$,  
\item[{\em iii)}] $\ell \cap \Sigma_{\alpha_{i}}$ is a Baer subline.
\end{itemize}

\begin{defin}\label{goodset}
    A set
\begin{align*}
& \cP = \left\{\left(P_{\alpha_{1} u_1}, \pi_{\alpha_{1} v_1}\right), \dots, \left(P_{\alpha_{{q+1}} u_{q+1}}, \pi_{\alpha_{{q+1}} v_{q+1}}\right)\right\}, & \alpha_{i} \in \cI, u_i, v_i \in \cU, 
\end{align*}
consisting of $q+1$ pairs, is said to be a {\em good set} if $\forall \alpha_{i}, \alpha_{j} \in \cI$, $\forall u_i, u_j, v_i, v_j \in \cU$, with $(\alpha_{i}, u_{i}, v_i) \ne (\alpha_{j}, u_j, v_j)$, it satisfies the following properties:
\begin{align}
\left(P_{\alpha_{i} u_i}, \pi_{\alpha_{i} v_i}\right), \left(P_{\alpha_{j} u_{j}}, \pi_{\alpha_{j} v_j}\right) \in \cP & \implies u_i v_j - v_i u_j \ne 0, \label{property1} \\
\left(P_{\alpha_{i} u_i}, \pi_{\alpha_{i} v_i}\right), \left(P_{\alpha_{j} u_{j}}, \pi_{\alpha_{j} v_{j}}\right) \in \cP & \implies \alpha_{i} u_i \left(\alpha_{j} v_{j}\right)^q - \left(\alpha_{i} v_i\right)^q \alpha_{j} u_{j} \ne 0. \label{property2}\end{align}
\end{defin}

We are now ready to show the main result of the paper. 

\begin{theorem}\label{main}
Let 
\begin{align*}
& \cP = \left\{\left(P_{\alpha_{1} u_1}, \pi_{\alpha_{1} v_1}\right), \dots, \left(P_{\alpha_{{q+1}} u_{q+1}}, \pi_{\alpha_{{q+1}} v_{q+1}}\right)\right\}, & \alpha_{i} \in \cI, u_i, v_i \in \cU, 
\end{align*}
and let 
\begin{align*}
\cL_{\cP} = \left\{p\left(P_{\alpha_{i} u_i}, \pi_{\alpha_{i} v_i}\right) \setminus \{r_{U_1}\} \mid \left(P_{\alpha_{i} u_i}, \pi_{\alpha_{i} v_i}\right) \in \cP \right\}. 
\end{align*}

The set 
\begin{align*}
\Pi_{\cP} = \left\{ \left(\cS_\ell \setminus \cR_{\ell}\right) \cup {\cR_{\ell}^o} \mid \ell \in \cL_{\cP} \right\} \cup \left\{\cD_{\eta}\right\} 
\end{align*}
is a line-parallelism of $\Sigma_{\eta}$ if and only if $\cP$ is a good set.
\end{theorem}
\begin{proof}
Let $\ell \in \cL_{\cP}$. By Lemma~\ref{lemma3}, $\cS_\ell$ is a Desarguesian line-spread of $\Sigma_{\eta}$ with transversal lines $\ell$, $\ell^{\tau_{\eta}}$. Since $\cD_{\eta} \cap \cS_\ell$ is a regulus $\cR_\ell$ and $\cR_\ell^o$ is its opposite regulus, it follows that $\left(\cS_\ell \setminus \cR_{\ell}\right) \cup \cR_{\ell}^{o}$ is a subregular line-spread of $\Sigma_{\eta}$ of index one. Next we prove that the $q^2+q+1$ involved line-spreads are pairwise disjoint if and only if $\cP$ is a good set. By construction
\begin{align*}
\cD_{\eta} \cap  \left((\cS_\ell \setminus \cR_{\ell}) \cup \cR_{\ell}^o\right) = \emptyset.
\end{align*}
Let $\ell_i, \ell_j \in \cL_{\cP}$, $\ell_i \ne \ell_j$, where $\ell_{i} \in p\left( P_{\alpha_{i} u_i}, \pi_{\alpha_{i} v_i} \right)$ and $\ell_j \in p\left( P_{\alpha_{j} u_{j}}, \pi_{\alpha_{j} v_{j}} \right)$, for some \\ $\left( P_{\alpha_{i} u_i}, \pi_{\alpha_{i} v_i} \right), \left( P_{\alpha_{j} u_{j}}, \pi_{\alpha_{j} v_{j}} \right) \in \cP$. We claim that 
\begin{align*}
\cR_{\ell_{i}}^o \cap \cR_{\ell_j}^o = \emptyset.
\end{align*}
Note that, by Proposition~\ref{mainp1}, the reguli $\cR_{\ell_{i}}$, $\cR_{\ell_j}$ are distinct if and only if \eqref{property1} holds true. If there is a line of $\Sigma_{\eta}$ contained in $\cR_{\ell_i}^o \cap \cR_{\ell_j}^o$, then there exists a line of $\cR_{\ell_{k}}$ and a line of $\cR_{\ell_j}$ having precisely one point, not belonging to $r_{U_1}$, in common, where $\cR_{\ell_{i}}$, $\cR_{\ell_j}$ are distinct reguli contained in the line-spread $\cD_{\eta}$, a contradiction. Also,
\begin{align*}
\left(\cS_{\ell_i} \setminus \cR_{\ell_i} \right) \cap \cR_{\ell_j}^o = \emptyset.
\end{align*}
Indeed, $\cS_{\ell_i}$ is a line-spread of $\Sigma_{\eta}$ containing $r_{U_1}$. Hence no line of $\cS_{\ell_i} \setminus \cR_{\ell_i}$ intersects $r_{U_1}$, whereas each line of $\cR_{\ell_j}^o$ intersects $r_{U_1}$ since $r_{U_1} \in \cR_{\ell_j}$. Finally, we claim that 
\begin{align*}
\left( \cS_{\ell_i} \setminus \cR_{\ell_i} \right) \cap \left( \cS_{\ell_j} \setminus \cR_{\ell_j} \right) = \emptyset.
\end{align*}
To this end, it is enough to show that $\ell_j$ has no points in common with the union of the lines of $\cS_{\ell_i} \setminus \cR_{\ell_i}$ extended over $\F_{q^2}$, that is
\begin{align*}
& \ell_j \cap \left( \overline{\cS}_{\ell_i} \setminus \overline{\cR}_{\ell_i} \right) = \emptyset,
\end{align*}
%If $(\alpha_{i_{k'}}, u_{k'}, v_{k'}) = (\alpha_{i_k}, u_k, v_k)$, then $\ell_j \cap \left( \overline{\cS}_{\ell_i} \setminus \overline{\cR}_{\ell_i} \right) = \emptyset$, by Proposition~\ref{mainp2}. Otherwise, if $(\alpha_{i_{k'}}, u_{k'}, v_{k'}) \ne (\alpha_{i_k}, u_k, v_k)$ and $\ell_j \cap \left( \overline{\cS}_{\ell_i} \setminus \overline{\cR}_{\ell_i} \right) \ne \emptyset$, then by Proposition~\ref{mainp2}
%\begin{align*}
%& \alpha_{i_k} u_k \left(\alpha_{i_{k'}} v_{k'}\right)^q - \left(\alpha_{i_k} v_k\right)^q \alpha_{i_{k'}} u_{k'} = 0,  %\mbox{ where } (\alpha_{i_{k'}}, u_{k'}, v_{k'}) \ne (\alpha_{i_k}, u_k, v_k),
%& \left(P_{\alpha_{i_k} u_k}, \pi_{\alpha_{i_k} v_k}\right), \left(P_{\alpha_{i_{k'}} u_{k'}}, \pi_{\alpha_{i_{k'}} v_{k'}}\right) = \left(P_{\frac{u_i}{u}}, \pi_{u v_i}\right) \in \cP, & \mbox{ where } u \in \cU, u \ne 1,
%\end{align*}
%contradicting 
By Proposition~\ref{mainp2}, we have that $\ell_j \cap \left( \overline{\cS}_{\ell_i} \setminus \overline{\cR}_{\ell_i} \right) = \emptyset$ if and only if 
\eqref{property2} holds true. Therefore
\begin{align*}
& \left( \cS_{\ell_i} \setminus \cR_{\ell_i} \right) \cap \left( \cS_{\ell_j} \setminus \cR_{\ell_j} \right) = \emptyset, & \mbox{ for } \ell_i, \ell_j \in \cL_{\cP}, \ell_i \ne \ell_j
\end{align*}
if and only if 
\eqref{property2} holds true. The proof is now complete.
\end{proof}

\subsection{A characterization}

Denote by $E$ the group of projectivities represented by 
\begin{align*}
    & \begin{pmatrix}
        1 & b & 0 & 0 \\
        0 & 1 & 0 & 0 \\
        0 & 0 & 1 & b^q \\
        0 & 0 & 0 & 1
    \end{pmatrix}, & b \in \F_{q^2}. 
\end{align*}
Then $E$ is an elementary abelian group of order $q^2$. It stabilizes $t_1$, $t_2$, and each of the subgeometries $\Sigma_{\alpha}$, $\alpha \in \Lambda$. Moreover, $E$ fixes the line $r_{U_1}$ pointwise and leaves invariant each of the planes through $r_{U_1}$. Therefore, 
\begin{align*}
& \left(p\left(P_{\alpha_i u_i}, \pi_{\alpha_i v_i}\right) \setminus \left\{ r_{U_1} \right\} \right)^g = p\left(P_{\alpha_i u_i}, \pi_{\alpha_i v_i}\right) \setminus \left\{ r_{U_1} \right\}, & g \in E,  \\
& \implies \left(p\left(P_{\alpha_i u_i}, \pi_{\alpha_i v_i}\right) \setminus \left\{ r_{U_1} \right\} \right)^E = p\left(P_{\alpha_i u_i}, \pi_{\alpha_i v_i}\right) \setminus \left\{ r_{U_1} \right\}. & 
\end{align*}
In particular, it can be easily checked that
\begin{align*}
    \ell^E & = \left\{ \ell^g \colon g \in E \right\} & \\
    & = p\left(P_{\alpha_i u_i}, \pi_{\alpha_i v_i}\right) \setminus \left\{ r_{U_1} \right\}, & \mbox{ if } \ell \in p\left(P_{\alpha_i u_i}, \pi_{\alpha_i v_i}\right) \setminus \left\{r_{U_1}\right\}. 
\end{align*}

\begin{prop}
    A line-parallelism $\Pi_{\cP}$ of $\Sigma_{\eta}$ admits $E$ as an automorphism group.
\end{prop}
\begin{proof}
    Recall that for a good set $\cP$, the line set $\cL_{\cP}$ is defined as follows
    \begin{align*}
    \cL_{\cP} = \left\{p\left(P_{\alpha_{i} u_i}, \pi_{\alpha_{i} v_i}\right) \setminus \{r_{U_1}\} \mid \left(P_{\alpha_{i} u_i}, \pi_{\alpha_{i} v_i}\right) \in \cP \right\}.
    \end{align*}
Hence 
\begin{align*}
    \cL_{\cP}^E & = \left\{\left(p\left(P_{\alpha_{i} u_i}, \pi_{\alpha_{i} v_i}\right) \setminus \{r_{U_1}\}\right)^E \mid \left(P_{\alpha_{i} u_i}, \pi_{\alpha_{i} v_i}\right) \in \cP \right\} \\
    & = \left\{p\left(P_{\alpha_{i} u_i}, \pi_{\alpha_{i} v_i}\right) \setminus \{r_{U_1}\} \mid \left(P_{\alpha_{i} u_i}, \pi_{\alpha_{i} v_i}\right) \in \cP \right\} \\
    & = \cL_{\cP}
\end{align*}
and 
\begin{align*}
\Pi_{\cP}^E & = \left\{ \left(\cS_\ell \setminus \cR_{\ell}\right) \cup {\cR_{\ell}^o} \mid \ell \in \cL_{\cP}^E \right\} \cup \left\{\cD_{\eta}\right\} \\
& = \left\{ \left(\cS_\ell \setminus \cR_{\ell}\right) \cup {\cR_{\ell}^o} \mid \ell \in \cL_{\cP} \right\} \cup \left\{\cD_{\eta}\right\} \\
& = \Pi_{\cP},
\end{align*}
as required.
\end{proof}
\begin{theorem}
Let $\Pi$ be a line-parallelism of $\Sigma_{\eta}$ comprising the Desarguesian spread $\cD_{\eta}$ and
$q^2 + q$ Hall spreads, which are obtained by switching the $q^2 + q$ reguli through
the fixed line $r_{U_1} \cap \Sigma_{\eta}$ of $\cD_{\eta}$. If $\Pi$ is left invariant by $E$, then there exists a good set $\cP$ such that $\Pi = \Pi_{\cP}$.
\end{theorem}
\begin{proof}
A Hall spread of $\Pi$ must be of type $\left(\cS_{\ell} \setminus \cR_{\ell}\right) \cup \cR_{\ell}^o$, where $\cR_{\ell}$ is a regulus of $\cD_{\eta}$ through the line $r_{U_1} \cap \Sigma_{\eta}$. By Lemma~\ref{lemma3}, 
we see that at least one among $\ell$ and $\ell^{\tau_\eta}$ belongs to $\cL$. Assume that $\ell \in \cL$, with $\ell \cap r_{U_1} = \{P_{\alpha_i u_i}\}$ and $\ell \subset \pi_{\alpha_i v_i}$. Then
\begin{align*}
    & \left(\cS_{\ell} \setminus \cR_{\ell}\right) \cup \cR_{\ell}^o \in \Pi, & \forall \ell \in p\left( P_{\alpha_i u_i}, \pi_{\alpha_i v_i} \right) \setminus \left\{r_{U_1}\right\},
\end{align*}
due to the fact that $\Pi$ is left invariant by $E$. Since $\left|p\left( P_{\alpha_i u_i}, \pi_{\alpha_i v_i} \right) \setminus \left\{r_{U_1}\right\}\right| = q$ and $\Pi$ contains $q^2+q$ Hall spreads, there is a set consisting of $q+1$ pairs:
\begin{align*}
& \cP = \left\{\left(P_{\alpha_{1} u_1}, \pi_{\alpha_{1} v_1}\right), \dots, \left(P_{\alpha_{{q+1}} u_{q+1}}, \pi_{\alpha_{{q+1}} v_{q+1}}\right)\right\}, & \alpha_{i} \in \cI, u_i, v_i \in \cU, 
\end{align*}
and a lineset
\begin{align*}
\cL_{\cP} = \left\{p\left(P_{\alpha_{i} u_i}, \pi_{\alpha_{i} v_i}\right) \setminus \{r_{U_1}\} \mid \left(P_{\alpha_{i} u_i}, \pi_{\alpha_{i} v_i}\right) \in \cP \right\},
\end{align*}
such that 
\begin{align*}
\Pi = \left\{ \left(\cS_\ell \setminus \cR_{\ell}\right) \cup {\cR_{\ell}^o} \mid \ell \in \cL_{\cP} \right\} \cup \left\{\cD_{\eta}\right\}.  
\end{align*}
The proof now follows from Theorem~\ref{main}.
\end{proof}

\section{The isomorphism issue}

Let $\cJ$ be the stabilizer in $\PGaL(4, q^2)$ of $\Sigma_{\eta}$. Two line-parallelisms $\Pi$ and $\Pi'$ of $\Sigma_{\eta}$ are said to be {\em equivalent} or {\em isomorphic} if there exists a collineation of $\cJ$ mapping $\Pi$ to $\Pi'$. In Theorem~\ref{main}, we have shown that given a good set $\cP$, it is possible to define a line set $\cL_{\cP}$ which in turn gives rise to the line-parallelism $\Pi_{\cP}$ of $\Sigma_{\eta}$. Of course, $\Pi_{\cP}$, $\Pi_{\cP'}$ may be equivalent, even if $\cP$ and $\cP'$ are distinct good sets. Here, we aim to obtain a bound on the number of non-equivalent line-parallelisms that can be constructed from Theorem~\ref{main}.

\subsection{The number of {\em good sets}}

Consider the following point sets of $\PG(2, q^2)$:
\begin{align*}
    & \cZ_{\alpha} = \left\{ (1, \alpha u, \alpha v) \mid u, v \in \cU \right\}, & \alpha \in \Lambda, \\
    & \cZ = \bigcup_{\alpha \in \cI} \cZ_{\alpha}.
\end{align*}
We note that each $\cZ_{\alpha}$ consists of $(q+1)^2$ points. Moreover, $\mathcal{Z}_{\alpha} \cap \mathcal{Z}_{\beta}=\emptyset$, when $\alpha \neq \beta$, and hence $|\cZ| = |\cI| (q+1)^2$. Let $\eps$ be the map sending a set 
\begin{align*}
    \cP = \left\{\left(P_{\alpha_{1} u_1}, \pi_{\alpha_{1} v_1}\right), \dots, \left(P_{\alpha_{{q+1}} u_{q+1}}, \pi_{\alpha_{{q+1}} v_{q+1}}\right)\right\}
\end{align*} 
consisting of $q+1$ pairs, where $\alpha_{i} \in \cI$, $u_i, v_i \in \cU$, to the subset of size $q+1$ of $\cZ$ given by:
\begin{align*}
    \eps({\cP}) = \left\{ (1, \alpha_{i} u_i, \alpha_{i} v_i) \mid \left(P_{\alpha_{i} u_i}, \pi_{\alpha_{i} v_i}\right) \in \cP \right\} \subset \cZ.
\end{align*}

Consider the following lines and conics  of $\PG(2, q^2)$:
\begin{align*}
    & s_c: X_2 = c X_3, & c \in \cU \\
    & C_{\alpha b}: \alpha^{q+1} b X_1^2 - X_2 X_3 = 0, & \alpha \in \Lambda, b \in \cU.
\end{align*}
Define the set $\cY$ of lines and the set $\cC_b$ of union of conics as follows
\begin{align*}
    & \cY = \left\{s_c \mid c \in \cU \right\}, & \\
    & \cC_{b} = \bigcup_{\alpha \in \cI} C_{\alpha b}, & b \in \cU.
\end{align*}
Observe that 
\begin{equation} \label{eq:belonglineconics}
    \left(1, \alpha u, \alpha v\right) \in s_{c} \cap C_{\alpha b},  \mbox{ where } c = \frac{u}{v}, b = \frac{\alpha u}{\alpha^q v^q}.
\end{equation}
Hence, for $\alpha \in \cI$, it can be easily checked that each of the $q+1$ members of $\cY$ or of $\left\{C_{\alpha b} \mid b \in \cU \right\}$ contains $q+1$ points of $\cZ_{\alpha}$, so that 
\begin{align*}
    & \left\{s_c \cap \cZ_{\alpha} \mid c \in \cU\right\}, \\
    & \left\{C_{\alpha b} \cap \cZ_{\alpha} \mid b \in \cU\right\}, 
\end{align*}
both are partitions of $\cZ_{\alpha}$.

The algebraic conditions that define a good set $\mathcal{P}$ can be interpreted geometrically on the associated point set $\eps(\mathcal{P})$ as follows.

\begin{prop} \label{prop3}
Let $\cP = \left\{\left(P_{\alpha_{1} u_1}, \pi_{\alpha_{1} v_1}\right), \dots, \left(P_{\alpha_{{q+1}} u_{q+1}}, \pi_{\alpha_{{q+1}} v_{q+1}}\right)\right\}$ be a set consisting of $q+1$ pairs, where $\alpha_{i} \in \cI$, $u_i, v_i \in \cU$. Then $\cP$ is a good set if and only if 
\begin{align}
    & |s_c \cap \eps(\cP)| = 1, & \forall c \in \cU, \label{prope1} \\
    & |\cC_{b} \cap \eps(\cP)| = 1, & \forall b \in \cU. \label{prope2}
\end{align}
\end{prop}
\begin{proof}
By \eqref{eq:belonglineconics}, we know that the point $(1, \alpha_{i} u_i, \alpha_{i} v_i)$ of $\eps(\cP)$ belongs to the line $s_{\frac{u_i}{v_i}}$ of $\cY$ and to the conic $C_{\alpha_{i} b} \subset \cC_b$, where 
\begin{equation} \label{eq:expressionbconic}
b = \frac{\alpha_{i} u_i}{\alpha_{i}^q v_i^q}.
\end{equation}
Now, let $(1, \alpha_{j} u_{j}, \alpha_{j} v_{j})$ a further point of $\eps(\cP)$, i.e., $\left(\alpha_{i}, u_i, v_i\right) \ne \left( \alpha_{j}, u_{j}, v_{j} \right)$. Then $(1, \alpha_{j} u_{j}, \alpha_{j} v_{j})$ lies in $s_{\frac{u_i}{v_i}}$ if and only if $u_i v_{j} - u_{j} v_i = 0$. Analogously, the point $(1, \alpha_{j} u_{j}, \alpha_{j} v_{j})$ lies in $C_{\beta b} \subset \cC_b$ if and only if
\[
\beta^{q+1} b - \alpha_j^2(u_j v_j) = 0.
\]
Raising both sides to the power $q+1$, we obtain
\[
\beta^{2(q+1)} = \alpha_j^{2(q+1)}.
\]
Since $\alpha_j \in \mathcal{I}$ and considering \eqref{eq:oppositenonbelong}, this implies $\beta = \alpha_j$. Thus, taking into account \eqref{eq:expressionbconic}, the point $(1, \alpha_{j} u_{j}, \alpha_{j} v_{j}) \in C_{\beta b}$ if and only if $\beta = \alpha_j$ and 
\[
\alpha_{i} u_i (\alpha_{j} v_{j})^q - (\alpha_{i} v_i)^q \alpha_{j} u_{j} = 0.
\]
On the other hand, we have 
    \begin{align*}
        \sum_{c \in \cU} \lvert s_c \cap \eps(\mathcal{P})\rvert = \lvert \eps(\mathcal{P})\rvert =q+1,
    \end{align*}
and 
    \begin{align*}
        \sum_{b \in \cU} \lvert \cC_b \cap \eps(\mathcal{P})\rvert = \lvert \eps(\mathcal{P})\rvert =q+1.
    \end{align*}
    Hence, $\cP$ satisfies \eqref{property1} and \eqref{property2} if and only if $|s_c \cap \eps(\cP)| = |\cC_b \cap \eps(\cP)| = 1$, $\forall c,b \in \cU$.
\end{proof}

Therefore, $\eps$ induces a bijection between good sets and subsets of $\cZ$ of size $q+1$ satisfying \eqref{prope1}, \eqref{prope2}. We now describe the intersection behavior of a line $s_c$ with a conic $\cC_{\alpha b}$ in relation to the $\cZ_{\beta}$'s. 

\begin{prop} \label{prop:geometricintersectiongood}
    Let $s_c$ be a line of $\cY$ and $C_{\alpha b}$ a conic contained in $\cC_b$. The size of the intersection $|s_c \cap C_{\alpha b} \cap \mathcal{Z}_\beta|$, with $\beta \in \cI$, is as follows. If $q$ is even:
    \begin{itemize}
        \item When $\alpha=\beta$, \begin{align*}
        |s_c \cap C_{\alpha b} \cap \cZ_{\alpha}| = 1.
    \end{align*}
    \item When $\alpha \neq \beta$, \begin{align*}
        |s_c \cap C_{\alpha b} \cap \cZ_{\beta}| = 0.
    \end{align*}
    \end{itemize}
%\begin{align*}
%        |s_a \cap C_{\alpha b} \cap \cZ_{\beta}| = \begin{cases}
 %           1 & \mbox{ if } \alpha = \beta, \\
%            0 & \mbox{ if } %\alpha \neq \beta.
%        \end{cases}
%    \end{align*}
If $q$ is odd:
\begin{itemize}
    \item When $\alpha=\beta$,
\begin{align*}
        |s_c \cap C_{\alpha b} \cap \cZ_{\alpha}| = \begin{cases}
            2 & \mbox{ if } \alpha^{q+1} \mbox{ is a square in } \F_{q} \mbox{ and } (cb)^{\frac{q+1}{2}}=1, \\
            2 & \mbox{ if } \alpha^{q+1} \mbox{ is not a square in } \F_{q} \mbox{ and } (cb)^{\frac{q+1}{2}}=-1, \\
            0 & \mbox{ otherwise},
        \end{cases}
    \end{align*}
%\item When $\alpha^{q+1}=-\beta^{q+1}$,     
%    \begin{align*}
 %       |s_c \cap C_{\alpha b} \cap \cZ_{\beta}| = \begin{cases}
 %           2 & \mbox{ if } \alpha^{q+1} \mbox{ is not a square in } \F_{q} \mbox{ and } (cb)^{\frac{q+1}{2}}=1, \\
  %          2 & \mbox{ if } \alpha^{q+1} \mbox{ is a square in } \F_{q} \mbox{ and } (cb)^{\frac{q+1}{2}}=-1, \\
  %          0 & \mbox{ otherwise},
 %       \end{cases}
 %   \end{align*}
    \item When $\alpha \neq \beta$,
    \begin{align*}
        |s_c \cap C_{\alpha b} \cap \cZ_{\beta}| = 0. 
    \end{align*}
\end{itemize}
\end{prop}
\begin{proof}
Let us start by considering the intersection $C_{\alpha b} \cap \mathcal{Z}_\beta$.
A point $(1, \beta u, \beta v) \in \mathcal{Z}_\beta$ lies on the conic $C_{\alpha b}$ if and only if
\begin{equation} \label{eq:equationintersectionconicssubgeo}
    \alpha^{q+1}b-\beta^2uv=0.
\end{equation}
Raising both sides to the power of $q+1$, we obtain
\begin{equation} \label{eq:raisingequationintersectionconicssubgeo}
    \alpha^{2(q+1)}=\beta^{2(q+1)}.
\end{equation}
Recall that $\alpha$ and $\beta$ belong to $\Lambda$, and hence either they are equal or have distinct norms. We now distinguish two cases based on the parity of $q$.

\textbf{Case 1: $q$ even}. 
From \eqref{eq:raisingequationintersectionconicssubgeo}, it follows that $\alpha = \beta$. Therefore, if $\alpha \ne \beta$, we conclude that $C_{\alpha b} \cap \cZ_{\beta}= \emptyset$ and thus
\[
\lvert s_c \cap C_{\alpha b} \cap \cZ_{\beta} \rvert=0.
\]
Now suppose $\alpha = \beta$. In this case, the point $(1, \alpha u, \alpha v) \in \mathcal{Z}_\alpha$ lies on the line $s_c$ if and only if $c = u/v$, which, combined with \eqref{eq:equationintersectionconicssubgeo}, implies that
\[
v^2=\frac{\alpha^{q-1}b}c.
\]
Since $q$ is even, the equation $x^2 = \alpha^{q-1}b/c$ has a unique solution in the field $\F_{q^2}$.
Hence, there exists a unique $v$, and consequently a unique $u$, such that the point $(1, \alpha u, \alpha v)$ lies in the intersection $s_c \cap C_{\alpha b}$.
Therefore, in this case, the intersection $s_c \cap C_{\alpha b} \cap \mathcal{Z}_\alpha$ consists of exactly one point. 

\textbf{Case 2: $q$ odd}. In this case, from \eqref{eq:raisingequationintersectionconicssubgeo}, we have that either $\alpha^{q+1} = \beta^{q+1}$ or $\alpha^{q+1} = -\beta^{q+1}$. Note that since $\alpha, \beta \in \mathcal{I}$, by \eqref{eq:oppositenonbelong}, the equality $\alpha^{q+1} = -\beta^{q+1}$ cannot occur. Therefore, we have $\alpha^{q+1} = \beta^{q+1}$, which implies $\alpha = \beta$.
 Therefore, it follows that $\alpha = \beta$. Hence, if $\alpha \ne \beta$, we conclude that $C_{\alpha b} \cap \mathcal{Z}_\beta = \emptyset$, and thus
\[
\lvert s_c \cap C_{\alpha b} \cap \mathcal{Z}_\beta \rvert = 0.
\]
So, assume $\alpha=\beta$. We now consider the intersection of the conic $C_{\alpha b}$ with the line $s_c$. Note that $\frac{\alpha^{q+1}b}{c}$ is a square in $\mathbb{F}_{q^2}$, so the intersection $C_{\alpha b} \cap s_c$ consists of the two points
\[
T_1=\left(1,\sqrt{\alpha^{q+1}cb}, \sqrt{\frac{\alpha^{q+1}b}c}\right) \mbox{ and }T_2=\left(1,-\sqrt{\alpha^{q+1}cb}, -\sqrt{\frac{\alpha^{q+1}b}c}\right),
\]
where $\sqrt{\frac{\alpha^{q+1}b}{c}}$ and $-\sqrt{\frac{\alpha^{q+1}b}{c}}$ are the two solutions in $\mathbb{F}_{q^2}$ of the equation $x^2 = \frac{\alpha^{q+1}b}{c}$.
%Also, we note that since $q$ is odd, we have $T_1 \in \mathcal{Z}_\alpha$ if and only if $T_2 \in \mathcal{Z}_\alpha$. Therefore, it suffices to determine when $T_1 \in \mathcal{Z}_\alpha$.
Also, since $q$ is odd, we note that 
\begin{align*}
    \left(\sqrt{\alpha^{q+1}cb}\right)^{q+1}= \left(\sqrt{\frac{\alpha^{q+1}b}c}\right)^{q+1}=\alpha^{q+1} \iff T_1 \in \cZ_{\alpha} \iff T_2 \in \cZ_{\alpha}.
\end{align*}
Observe first that, since $(cb)^{q+1} = (c/b)^{q+1} = 1$, it follows that either $(cb)^{\frac{q+1}{2}}=(c/b)^{\frac{q+1}{2}}=1$ or $(cb)^{\frac{q+1}{2}}=(c/b)^{\frac{q+1}{2}}=-1$. Moreover, when $\alpha^{q+1}$ is a square in $\mathbb{F}_q$, it is easy to check that $\left(\sqrt{\alpha^{q+1}}\right)^{q+1}=\alpha^{q+1}$, while if $\alpha^{q+1}$ is not a square in $\mathbb{F}_q$, then $\left(\sqrt{\alpha^{q+1}}\right)^{q+1}=-\alpha^{q+1}$. Combining these observations, we obtain:
\[
\left(\sqrt{\alpha^{q+1}cb}\right)^{q+1}= \left(\sqrt{\frac{\alpha^{q+1}b}c}\right)^{q+1}=\alpha^{q+1},
\]
when either $\alpha^{q+1}$ is a square in $\mathbb{F}_q$ and $(cb)^{\frac{q+1}{2}} = 1$, or $\alpha^{q+1}$ is not a square in $\mathbb{F}_q$ and $(cb)^{\frac{q+1}{2}} = -1$. The assertion immediately follows.
%On the other hand,
%%\[\left(\sqrt{\alpha^{q+1}cb}\right)^{q+1}=\left(\sqrt{\frac{\alpha^{q+1}b}c}\right)^{q+1}=-\alpha^{q+1},\]
%when either $\alpha^{q+1}$ is not a square in $\mathbb{F}_q$ and $(cb)^{\frac{q+1}{2}} = 1$, or $\alpha^{q+1}$ is a square in $\mathbb{F}_q$ and $(cb)^{\frac{q+1}{2}} = -1$. 
\end{proof}

When $q$ is odd, define
\[\cI_1:=\left\{\alpha \in \cI \colon \alpha^{q+1} \mbox{ is a square in } \F_q\right\} \ \ \  \mbox{ and } \ \ \ \cI_2:=\left\{\alpha \in \cI \colon \alpha^{q+1} \mbox{ is not a square in }\F_q\right\}\]

\begin{lemma} We have 
\[\lvert 
\cI_1 \rvert= 
\begin{cases}
 \frac{q-1}{4} & \mbox{ if }q \equiv 1 \pmod{4}, \\ 
 \frac{q-3}{4} & \mbox{ if } q \equiv 3 \pmod{4}, 
 \end{cases}\] and  
 \[\lvert \cI_2 \rvert= 
 \begin{cases} 
 \frac{q-1}{4} & \mbox{ if } q \equiv 1 \pmod{4}, \\  
 \frac{q+1}{4} & \mbox{ if } q \equiv 3 \pmod{4}. 
 \end{cases}\] 
 \end{lemma}

\begin{proof}
    Clearly, $\mathcal{I} = \mathcal{I}_1 \cup \mathcal{I}_2$ and $\mathcal{I}_1 \cap \mathcal{I}_2 = \emptyset$, so \begin{equation} \label{eq:sumI1I2}
        \lvert \cI_1\rvert+\lvert \cI_2\rvert=\lvert \cI \rvert=\frac{q-1}{2}.\end{equation}
    Let $A = \left\{a_1, \ldots, a_{\frac{q-3}{2}}\right\}$ and $A^{-1} = \left\{a_1^{-1}, \ldots, a_{\frac{q-3}{2}}^{-1}\right\}$, and by Lemma~\ref{lm:partitionnorm}, we have $\{\pm 1\} \cup A \cup A^{-1} = \F_q \setminus \{0\}$. 
Note that a nonzero element $a \in \mathbb{F}_q$ is a square if and only if its inverse $a^{-1}$ is also a square in $\mathbb{F}_q$. Therefore, the number of squares in $A$ equals to the number of squares in $A^{-1}$. Recall that $-1$ is a square in $\mathbb{F}_q$ if and only if $q \equiv 1 \pmod{4}$. Also, the number of nonzero squares in $\mathbb{F}_q$ is $\frac{q-1}{2}$. Thus: 
\begin{itemize}
    \item If $q \equiv 1 \pmod{4}$, it follows that the number of squares in $A$ is $\frac{1}{2} \left( \frac{q-1}{2} - 2 \right) = \frac{q-5}{4}$.

\item If $q \equiv 3 \pmod{4}$, then the number of squares in $A$ is $\frac{1}{2} \left( \frac{q-1}{2} - 1 \right) = \frac{q-3}{4}$.
\end{itemize}
Finally, recall that the element $\beta \in \Lambda$ such that $\beta^{q+1} = -1$ belongs to $\mathcal{I}$. Therefore, the size of $\mathcal{I}_1$ equals the number of squares in $A$ plus one when $q \equiv 1 \pmod{4}$, and the size of $A$ when $q \equiv 3 \pmod{4}$. Taking into account equation~\eqref{eq:sumI1I2}, the claim follows immediately.
\end{proof}

\begin{cor} \label{cor:totalintersection} For fixed $b,c \in \cU$. the following hold. If $q$ is even, then 
\[
 |s_c \cap \cC_{b} \cap \mathcal{Z}|=\lvert \cI \rvert=\frac{q-2}{2}.
\]
If $q$ is odd, then 
   \begin{align*}
      |s_c \cap \cC_{b} \cap \mathcal{Z}|= \begin{cases}
           2 \lvert \cI_1 \rvert & \mbox{ if } (cb)^{\frac{q+1}{2}}=1, \\
           2 \lvert \cI_2 \rvert & \mbox{ if } (cb)^{\frac{q+1}{2}}=-1.
       \end{cases}
    \end{align*}
\end{cor}

\begin{proof}
    First, consider the case where $q$ is even. By Proposition \ref{prop:geometricintersectiongood}, we immediately obtain
    \[
    |s_c \cap \cC_{b} \cap \mathcal{Z}|=\sum_{\alpha \in \cI} |s_c \cap C_{\alpha b} \cap \mathcal{Z}_{\alpha}|=\lvert \cI \rvert=\frac{q-2}{2}. 
    \]
    Now, suppose that $q$ is odd. Again, by Proposition \ref{prop:geometricintersectiongood}, we have
     \[
     \begin{array}{rl}
    |s_c \cap \cC_{b} \cap \mathcal{Z}| & =\sum\limits_{\alpha \in \cI} |s_c \cap C_{\alpha b} \cap \mathcal{Z}_{\alpha}|\\
    &
    =\sum\limits_{\alpha \in \cI_1} |s_c \cap C_{\alpha b} \cap \mathcal{Z}_{\alpha}| + \sum\limits_{\alpha \in \cI_2} |s_c \cap C_{\alpha b} \cap \mathcal{Z}_{\alpha}|\\
    & =\begin{cases}
           2 \lvert \cI_1 \rvert & \mbox{ if } (cb)^{\frac{q+1}{2}}=1, \\
           2 \lvert \cI_2 \rvert & \mbox{ if } (cb)^{\frac{q+1}{2}}=-1.
       \end{cases} 
     \end{array}
    \]
    This completes the proof.
\end{proof}

We now count the number of good sets. We start with the case where $q$ is even. 

\begin{prop}\label{prop:allgoodsetseven}
    Assume that $q$ is even. The number of good sets equals  
    \begin{align*}
    \lvert \cI \rvert^{q+1} (q+1)!=\left( \frac{q-1}{2} \right)^{q+1} \cdot (q+1)!
    \end{align*}
\end{prop}
\begin{proof}
By Proposition~\ref{prop3}, the number of good sets is equal to the number of subsets $\varepsilon(\mathcal{P}) \subset \mathcal{Z}$ of size $q+1$ such that $|s_c \cap \varepsilon(\mathcal{P})| = |\cC_b \cap \varepsilon(\mathcal{P})| = 1$, $\forall c, b \in \mathcal{U}$. For each $b \in \mathcal{U}$, observe that $C_{\alpha b} \cap C_{\beta b} = \emptyset$, if $\alpha, \beta \in \cI$, $\alpha \ne \beta$. Set $\eps(\cP) = \{R_1, \dots, R_{q+1}\}$ and let $R_i$ denote the unique point in $\varepsilon(\mathcal{P})$ lying on $\cC_{b_i}$, where $\mathcal{U} = \{b_1, \ldots, b_{q+1}\}$. Then $R_i \in C_{\alpha_i b_i}$ for some $\alpha_i \in \mathcal{I}$. Moreover, by Proposition \ref{prop:geometricintersectiongood}, each line in $\mathcal{Y}$ intersects a given conic $C_{\alpha_i b_i}$ in exactly one point. Thus, for every $R_i$ there exists a unique line $s_{c_i} \in \mathcal{Y}$ such that $R_i \in s_{c_i}$, where $\mathcal{U} = \{c_1, \ldots, c_{q+1}\}$.
%Let $\mathcal{U} = \{b_1, \ldots, b_{q+1}\}$. For each $b \in \mathcal{U}$, observe that $C_{\alpha b} \cap C_{\beta b} = \emptyset$, if $\alpha \ne \beta$. Since $|\cC_b \cap \varepsilon(\mathcal{P})| = 1$ for all $b \in \mathcal{U}$, let $R_i$ denote the unique point in $\varepsilon(\mathcal{P})$ lying on $\cC_{b_i}$. Then $R_i \in C_{\alpha_i b_i}$ for some $\alpha_i \in \mathcal{I}$. Moreover, by Proposition \ref{prop:geometricintersectiongood}, each line in $\mathcal{Y}$ intersects a given conic $C_{\alpha_i b_i}$ in exactly one point. Thus, for every $R_i$ there exists a unique line $s_{c_i} \in \mathcal{Y}$ such that $R_i \in s_{c_i}$. 
It follows that $R_1$ can be chosen in $|\mathcal{I}| \cdot (q+1)$ ways and, by Corollary \ref{cor:totalintersection}, $R_i$ (for $i > 1$) can be chosen in
\[
|\mathcal{I}| \cdot (q+1) - \sum_{j=1}^{i-1} |s_{c_j} \cap \cC_{b_j} \cap \mathcal{Z}| = |\mathcal{I}| \cdot (q+1 - (i - 1))
\]
ways. This establishes the claim.
\end{proof}

Now we study the case $q$ odd.

\begin{prop}\label{prop:allgoodsetsodd}
Assume that $q$ is odd. The number of good sets equals  
    \begin{align*}
    \left(\lvert \cI_1 \rvert \lvert \cI_2 \rvert\right)^{\frac{q+1}{2}} \prod_{i=0}^{\frac{q-1}{2}} \left(q+1-2i \right)^2=
    \begin{cases} 
    \left(\frac{q-1}{4} \right)^{q+1} \prod\limits_{i=0}^{\frac{q-1}{2}} \left(q+1-2i \right)^2 & \mbox{ if }q \equiv 1 \pmod{4}, \\ 
    \left(\frac{(q+1)(q-3)}{16}\right)^{\frac{q+1}{2}} \prod\limits_{i=0}^{\frac{q-1}{2}} \left(q+1-2i \right)^2 & \mbox{ if } q \equiv 3 \pmod{4}. 
    \end{cases}
    \end{align*}
\end{prop}

\begin{proof}
Again, by Proposition~\ref{prop3}, the number of good sets equals the number of subsets $\varepsilon(\mathcal{P}) \subset \mathcal{Z}$ of size $q+1$ such that $|s_c \cap \eps(\cP)| = |\cC_b \cap \eps(\cP)| = 1$, $\forall c,b \in \cU$. Again, for each $b \in \mathcal{U}$, observe that $C_{\alpha b} \cap C_{\beta b} = \emptyset$, if $\alpha, \beta \in \cI$, $\alpha \ne \beta$. Set $\eps(\cP) = \{R_1, \dots, R_{q+1}\}$ and let $R_i$ denote the unique point in $\varepsilon(\mathcal{P})$ lying on $\cC_{b_i}$, where $\mathcal{U} = \{b_1, \ldots, b_{q+1}\}$. Also, through each point $R_i$ there passes exactly one line of $\mathcal{Y}$, say $s_{c_i}$, for $i \in \{1, \ldots, q+1\}$ and $\cU=\{c_1,\ldots,c_{q+1}\}$. Without loss of generality, we may assume that 
\begin{align*}
b_i^{\frac{q+1}2} & =1, & \mbox{ for } i \in \left\{1,\ldots,\frac{q+1}2\right\}, \\
b_i^{\frac{q+1}2} & =-1, & \mbox{ for } i \in \left\{\frac{q+1}2+1,\ldots,q+1\right\}.
\end{align*}
Then
\begin{align*}
& R_i \in C_{\alpha_i b_i}, & \text{for some } \alpha_i \in \mathcal{I}, i \in \left\{1, \ldots, \frac{q+1}{2} \right\}, \\
& R_{\frac{q+1}{2}+i} \in C_{\beta_i b_{\frac{q+1}{2}+i}}, & \quad \text{for some } \beta_i \in \mathcal{I}, i \in \left\{1, \ldots, \frac{q+1}{2} \right\}. 
\end{align*}
We may assume that
\[
\alpha_1, \ldots, \alpha_h, \ \beta_1, \ldots, \beta_k \in \mathcal{I}_1
\quad \text{and} \quad
\alpha_{h+1}, \ldots, \alpha_{\frac{q+1}{2}}, \ \beta_{k+1}, \ldots, \beta_{\frac{q+1}{2}} \in \mathcal{I}_2,
\]
for some integers \( h, k \in \left\{0, 1, \ldots, \frac{q+1}{2} \right\} \). With this notation if \( h = 0 \), then all \( \alpha_1, \ldots, \alpha_{\frac{q+1}{2}} \in \mathcal{I}_2 \), whereas if \( h = \frac{q+1}{2} \), then all \( \alpha_1, \ldots, \alpha_{\frac{q+1}{2}} \in \mathcal{I}_1 \). Similarly, if \( k = 0 \), then all \( \beta_1, \ldots, \beta_{\frac{q+1}{2}} \in \mathcal{I}_2 \), whereas if \( k = \frac{q+1}{2} \), then all \( \beta_1, \ldots, \beta_{\frac{q+1}{2}} \in \mathcal{I}_1 \). Observe that Corollary~\ref{cor:totalintersection} implies that
\begin{align*}
& c_i^{\frac{q+1}{2}} = 1, & \text{for all } i \in \{1, \ldots, h\} \cup \left\{k + \frac{q+1}{2} + 1, \ldots, q+1\right\}, \\
& c_i^{\frac{q+1}{2}} = -1, & \text{for all } i \in \left\{h+1, \ldots, \frac{q+1}{2} \right\} \cup \left\{ \frac{q+1}{2} + 1, \ldots, k + \frac{q+1}{2} \right\}.
\end{align*}
Since the number of elements $c \in \cU$ such that $c^{\frac{q+1}{2}} = 1$ is exactly $(q+1)/2$, it follows that
\[
h = k.
\]
Consider first the case where $i \in \{1, \ldots, h\} \cup \left\{h + \frac{q+1}{2} + 1, \ldots, q+1\right\}$. By Proposition \ref{prop:geometricintersectiongood}, we know that the line $s_{c_i}$ is secant to the conic $\mathcal{C}_{\alpha_i b_i}$. Therefore, $R_1$ can be chosen in $|\mathcal{I}_1| \cdot (q+1)$ ways. Also, by Corollary \ref{cor:totalintersection}, for $i \in \{2, \ldots, h\}$, the point $R_i$ can be chosen in
\[
|\mathcal{I}_1| \cdot (q+1) - \sum_{j=1}^{i-1} 2|\mathcal{I}_1| = |\mathcal{I}_1| \cdot (q+1 - 2(i - 1))
\]
ways. Similarly, for $i \in \left\{h + \frac{q+1}{2} + 1, \ldots, q+1\right\}$, again by Corollary \ref{cor:totalintersection}, $R_i$ can be chosen in
\[
|\mathcal{I}_1| \cdot (q+1) - \sum_{j=1}^{h - 1} 2|\mathcal{I}_1| - \sum_{j=h + \frac{q+1}{2} + 1}^{i} 2|\mathcal{I}_1| = |\cI_1| \cdot \left(q+1 - 2 \left(i - h + \frac{q+1}{2}\right)\right)
\]
ways. Now, the choice of $R_i$ for indices 
\[
i \in \left\{h+1, \ldots, \frac{q+1}{2} \right\} \cup \left\{ \frac{q+1}{2} + 1, \ldots, h + \frac{q+1}{2} \right\}
\]
is independent of the choices of $R_1, \ldots, R_h$ and $R_{h + \frac{q+1}{2} + 1}, \ldots, R_{q+1}$. Therefore, similarly to the previous case, for each $i \in \left\{h+1,\ldots,\frac{q+1}{2}\right\}$, there are
\[
|\mathcal{I}_2| \cdot \left( (q+1) - 2(i-h-1) \right)
\]
possible choices for $R_i$, whereas there are
\[
|\mathcal{I}_2| \cdot \left( (q+1) - 2(i-h-1) \right)
\]
possible ways to choose $R_i$ if $i \in \left\{\frac{q+1}{2}+1,\ldots,h+\frac{q+1}{2}\right\}$. This completes the proof of the statement.
     \end{proof} 

As we will see in Section \ref{subse:automorphis}, in order to estimate the number of inequivalent parallelisms, it is useful to count certain good sets.% such that the $\alpha_i$ satisfy $\alpha_i^{q+1} \neq -1$. For this goal, the same argument used in Proposition \ref{prop:allgoodsetsodd} can be adapted, leading to the following result.

\begin{cor}\label{cor:allgoodsetsodd}
Assume that $q$ is odd. The number of good sets 
\begin{align*}
& \cP = \left\{\left(P_{\alpha_{1} u_1}, \pi_{\alpha_{1} v_1}\right), \dots, \left(P_{\alpha_{{q+1}} u_{q+1}}, \pi_{\alpha_{{q+1}} v_{q+1}}\right)\right\},  
\end{align*}
with $\alpha_i^{q+1} \neq -1$, $i = 1, \dots, q+1$, equals  
    \begin{align*}
    & \left(\frac{(q-5)(q-1)}{16}\right)^{\frac{q+1}{2}} \prod_{i=0}^{\frac{q-1}{2}} \left(q+1-2i \right)^2 & \mbox{ if }q \equiv 1 \pmod{4}, \\ 
    & \left(\frac{q-3}{4}\right)^{q+1} \prod_{i=0}^{\frac{q-1}{2}} \left(q+1-2i \right)^2 & \mbox{ if } q \equiv 3 \pmod{4}. 
    \end{align*}
\end{cor}

\begin{proof}
The assertion follows immediately by applying the same argument as in Proposition \ref{prop:allgoodsetsodd}, and by observing that if $q \equiv 1 \pmod{4}$, then the element $\beta \in \mathcal{I}$ such that $\beta^{q+1}=-1$ belongs to $\mathcal{I}_1$, whereas if $q \equiv 3 \pmod{4}$, then $\beta$ belongs to $\mathcal{I}_2$.
\end{proof} 

\subsection{The automorphism group of \texorpdfstring{$\Pi_{\cP}$}{Pi}} \label{subse:automorphis}

Let $\Gamma$ be the stabilizer of the Desarguesian line-spread $\cD_{\eta}$ in $\cJ$, where we recall that $\cJ$ is the stabilizer in $\PGaL(4, q^2)$ of $\Sigma_{\eta}$. Note that $\Gamma = Stab_{\cJ}\left(\overline{\cD}_{\eta}\right)$, and from \cite{dye1991spreads}, \cite{van2016desarguesian}, we have that 
\begin{align*}
    \Gamma \simeq \left(\GL(2, q^2)/Z \rtimes \iota \right) \rtimes Aut(\F_q),
\end{align*}
where $\iota$ is represented by 
\begin{align*}
\begin{pmatrix}
0 & 0 & 1 & 0 \\
0 & 0 & 0 & 1 \\
1 & 0 & 0 & 0 \\
0 & 1 & 0 & 0 
\end{pmatrix},
\end{align*}
and 
\begin{align*}
    & \begin{pmatrix}
        a & b & 0 & 0 \\
        c & d & 0 & 0 \\
        0 & 0 & a^q & b^q \\
        0 & 0 & c^q & d^q
    \end{pmatrix}, & a,b,c,d \in \F_{q^2}, ad-bc \ne 0, \\
    & \begin{pmatrix}
        \lambda & 0 & 0 & 0 \\
        0 & \lambda & 0 & 0 \\
        0 & 0 & \lambda & 0 \\
        0 & 0 & 0 & \lambda
    \end{pmatrix}, & \lambda \in \F_{q}, 
\end{align*}
are the elements of $\GL(2, q^2)$ and $Z$, respectively. Then 
\begin{align*}
    \lvert \Gamma \rvert = 2hq^2(q^4-1)(q+1).
\end{align*}
Furthermore, the subgroup of $\Gamma$ that stabilizes the line $r_{U_1}$, namely  $\Gamma_{r_{U_1}}$, has size
\begin{align*}
    \lvert \Gamma_{r_{U_1}} \rvert = 2hq^2(q^2-1)(q+1),
\end{align*}
since $r_{U_1} \cap \Sigma_{\eta} \in \cD_{\eta}$ and $\Gamma$ is transitive on the $q^2+1$ lines of $\cD_{\eta}$.

\begin{prop}\label{automorphisms}
Let $\Pi_{\cP}$, $\Pi_{\cP'}$ line-parallelisms of $\Sigma_{\eta}$, where $\cP$, $\cP'$ are good sets. If there exists $ \psi \in \cJ$ such that $\Pi_{\cP}^\psi = \Pi_{\cP'}$, then $\psi \in \Gamma_{r_{U_1}}$. 
\end{prop}
\begin{proof}
By construction 
\begin{align*}
& \Pi_{\cP} = \left\{ \left(\cS_\ell \setminus \cR_{\ell}\right) \cup {\cR_{\ell}^o} \mid \ell \in \cL_{\cP} \right\} \cup \left\{\cD_{\eta}\right\}, \\
& \Pi_{\cP'} = \left\{ \left(\cS_{\ell'} \setminus \cR_{\ell'}\right) \cup {\cR_{\ell'}^o} \mid \ell' \in \cL_{\cP'} \right\} \cup \left\{\cD_{\eta}\right\}.
\end{align*}
If $\psi \in \cJ$ is such that $\Pi_{\cP}^\psi = \Pi_{\cP'}$, then $\cD_{\eta}^\psi = \cD_{\eta}$, since $\cD_{\eta}$ is the unique Desarguesian line-spread in $\Pi_{\cP}$ and $\Pi_{\cP'}$. Hence $\psi \in \Gamma$. Let $\ell, \tilde{\ell}$ be distinct lines of $\cL_{\cP}$ and $\ell', \tilde{\ell'}$ be distinct lines of $\cL_{\cP'}$ such that 
\begin{align*}
    & \left(\left(\cS_\ell \setminus \cR_{\ell}\right) \cup {\cR_{\ell}^o}\right)^\psi = \left(\cS_{\ell'} \setminus \cR_{\ell'}\right) \cup {\cR_{\ell'}^o}, \\
    & \left(\left(\cS_{\tilde{\ell}} \setminus \cR_{\tilde{\ell}}\right) \cup {\cR_{\tilde{\ell}}^o}\right)^\psi = \left(\cS_{\tilde{\ell'}} \setminus \cR_{\tilde{\ell'}}\right) \cup {\cR_{\tilde{\ell'}}^o}.
\end{align*}
Then necessarily 
\begin{align*}
    & \cR_{\ell}^\psi = \cR_{\ell'}, \\
    & \cR_{\tilde{\ell}}^\psi = \cR_{\tilde{\ell'}},
\end{align*}
where 
\begin{align*}
    r_{U_1} \cap \Sigma_{\eta} = \cR_{\ell} \cap \cR_{\tilde{\ell}} = \cR_{\ell'} \cap \cR_{\tilde{\ell'}}.
\end{align*}
Therefore, 
\begin{align*}
    \left(r_{U_1} \cap \Sigma_{\eta}\right)^\psi & = \left( \cR_{\ell} \cap \cR_{\tilde{\ell}} \right)^\psi \\
    & = \cR_{\ell}^\psi \cap \cR_{\tilde{\ell}}^\psi \\
    & = \cR_{\ell'} \cap \cR_{\tilde{\ell'}} \\
    & = r_{U_1} \cap \Sigma_{\eta},
\end{align*}
and hence $\psi \in Stab_{\Gamma}\left(r_{U_1}\right) = \Gamma_{r_{U_1}}$.
\end{proof}

\begin{cor}
    $Stab_{\cJ}(\Pi_{\cP}) \le \Gamma_{r_{U_1}}$.
\end{cor}

Denote by $\cT$ the set of all line-parallelisms of $\Sigma_{\eta}$, obtained from the good sets by means of Theorem~\ref{main}, i.e.,
\begin{align*}
    \cT = \left\{ \Pi_{\cP} \mid \cP \mbox{ is a good set}\right\}.
\end{align*}

By construction, a line $\ell \in \cL_{\cP}$ uniquely identifies the point $P_{\alpha_iu_i} = \ell \cap r_{U_1}$ and the plane $\pi_{\alpha_iv_i}$ spanned by $\ell$ and $r_{U_1}$. Hence, if $\cP$, $\cP'$ are good sets, then 
\begin{align*}
    \cL_{\cP} = \cL_{\cP'} \iff \cP = \cP'.
\end{align*}
Denote by 
\begin{align*}  \cL_{\cP}^{\tau_{\eta}} & = \left\{ \ell^{\tau_{\eta}} \mid \ell \in \cL_{\cP} \right\} \\
    & = \left\{p\left(P_{\alpha_iu_i}^{\tau_{\eta}}, \pi_{\alpha_iv_i}^{\tau_{\eta}}\right) \setminus \{r_{U_1}\} \mid \left(P_{\alpha_iu_i}, \pi_{\alpha_iv_i}\right) \in \cP \right\}. 
\end{align*}
\begin{prop}\label{distinctgoodsets}
%    Assume that $\alpha_2^{q+1} \ne -1$, in the case when $q$ is odd. Then 
%\begin{align*}
Let \begin{align*}
& \cP = \left\{\left(P_{\alpha_{1} u_1}, \pi_{\alpha_{1} v_1}\right), \dots, \left(P_{\alpha_{{q+1}} u_{q+1}}, \pi_{\alpha_{{q+1}} v_{q+1}}\right)\right\}, & \alpha_{i} \in \cI, u_i, v_i \in \cU, 
\end{align*}
and 
\begin{align*}
& \cP' = \left\{\left(P_{\alpha_{1}' u_1'}, \pi_{\alpha_{1}' v_1'}\right), \dots, \left(P_{\alpha_{q+1}' u'_{q+1}}, \pi_{\alpha_{q+1}' v'_{q+1}}\right)\right\}, & \alpha_{i}' \in \cI, u_i', v_i' \in \cU, 
\end{align*}
be good sets. Assume that $\alpha_i^{q+1}\neq -1$ and $\alpha_i'^{q+1}\neq -1$, $i = 1, \dots, q+1$. If $\cP \ne \cP'$, then $\Pi_{\cP} \ne \Pi_{\cP'}$.
%\end{align*} 
\end{prop}
\begin{proof}
    If $\Pi_{\cP} = \Pi_{\cP'}$, then necessarily 
    \begin{align}
        \cL_{\cP} \cup \cL_{\cP}^{\tau_{\eta}} = \cL_{\cP'} \cup \cL_{\cP'}^{\tau_{\eta}}. \label{eq}
    \end{align}
The lines in $\cL_{\cP}$, $\cL_{\cP'}$ meet certain subgeometries in $\left\{\Sigma_{\alpha_{i}} \mid \alpha_{i} \in \cI\right\}$, in a Baer subline, whereas, by Lemma~\ref{lemmasub}, those in $\cL_{\cP}^{\tau_{\eta}}$, $\cL_{\cP'}^{\tau_{\eta}}$ meet some subgeometries in $\left\{\Sigma_{\alpha_{j}} \mid \alpha_{j} \not\in \cI\right\}$, in a Baer subline. Therefore $\Sigma_{\alpha_i} \cap \Sigma_{\alpha_j} = \emptyset$, where $\alpha_i \in \cI$, $\alpha_j \notin \cI$. Then, \eqref{eq} implies 
$\cL_{\cP} = \cL_{\cP'}$,
that is $\cP = \cP'$.
\end{proof}
%The previous result is no longer true if $q$ is odd and $\alpha_2^{q+1} = -1$, see for instance the case $q=3$ described in Section~\ref{caseq=3}. 

\begin{theorem}
The number of mutually not equivalent line-parallelisms of $\Sigma_{\eta}$ contained in $\cT$ is at least
    \begin{align*}
      & \left(\frac{q-1}{2}\right)^{q+1} \frac{(q-2)!}{2hq(q+1)} & \mbox{ if } q \mbox{ is even, } \\
      & \left(\frac{(q-5)(q-1)}{16}\right)^{\frac{q+1}{2}} \frac{(q^2-1)^{\frac{q-1}{2}}}{2hq^2(q+1)} & \mbox{ if }q \equiv 1 \pmod{4}, \\ 
      &  \left(\frac{q-3}{4}\right)^{q+1} \frac{(q^2-1)^{\frac{q-1}{2}}}{2hq^2(q+1)} & \mbox{ if } q \equiv 3 \pmod{4}.  
    \end{align*}
\end{theorem}
\begin{proof}
From Proposition~\ref{automorphisms}, if $\Pi_{\cP}, \Pi_{\cP'} \in \cT$ and $\Pi_{\cP}^\psi = \Pi_{\cP'}$, for some $\psi \in \cJ$, then $\psi \in \Gamma_{r_{U_1}}$. Hence
\begin{equation} \label{eq:boundstabiliser}
|\Pi_{\cP}^{\cJ} \cap \cT| \le |\Gamma_{r_{U_1}}| = 2hq^2(q^2-1)(q+1)
\end{equation}

Assume first that $q$ is even. By the definition of good set (see Definition \ref{goodset}), we know that each good set
\begin{align*}
\cP = \left\{\left(P_{\alpha_{1} u_1}, \pi_{\alpha_{1} v_1}\right), \dots, \left(P_{\alpha_{q+1} u_{q+1}}, \pi_{\alpha_{q+1} v_{q+1}}\right)\right\}
\end{align*}
is such that $\alpha_i^{q+1} \neq 1$, $i = 1, \dots, q+1$. Therefore, by Proposition \ref{distinctgoodsets}, the number of distinct parallelisms in $\cT$ equals the number of distinct good sets. By Proposition \ref{prop:allgoodsetseven}, this number is given by:
\[
|\cT| = |\cI|^{q+1} \cdot (q+1)! = \left( \frac{q-1}{2} \right)^{q+1} \cdot (q+1)!.
\]
Now, considering \eqref{eq:boundstabiliser}, we obtain that in $\cT$ there are at least
\[
\begin{array}{rl}
\displaystyle \frac{|\cT|}{|\Pi_{\cP}^{\cJ} \cap \cT|} &\displaystyle \ge \left( \frac{q-1}{2} \right)^{q+1} \cdot \frac{(q+1)!}{2hq^2(q^2 - 1)(q+1)} \\[10pt]
&\displaystyle = \left( \frac{q-1}{2} \right)^{q+1} \cdot \frac{(q-2)!}{2hq(q + 1)}
\end{array}
\]
non-equivalent line-parallelisms of $\Sigma_{\eta}$.

Now, assume that $q$ is odd. The size of $\cT$ is clearly greater than or equal to the number of good sets
\begin{align*}
\cP = \left\{\left(P_{\alpha_{1} u_1}, \pi_{\alpha_{1} v_1}\right), \dots, \left(P_{\alpha_{q+1} u_{q+1}}, \pi_{\alpha_{q+1} v_{q+1}}\right)\right\},
\end{align*}
having the property that $\alpha_i^{q+1} \neq -1$, $i = 1, \dots, q+1$. Therefore, by Corollary \ref{cor:allgoodsetsodd} and taking into account Proposition \ref{distinctgoodsets}, we obtain
\[
|\cT| \geq 
    \begin{cases} 
    \left(\frac{(q-5)(q-1)}{16}\right)^{\frac{q+1}{2}} \prod\limits_{i=0}^{\frac{q-1}{2}} \left(q+1-2i \right)^2 & \mbox{ if }q \equiv 1 \pmod{4}, \\ 
    \left(\frac{q-3}{4}\right)^{q+1} \prod\limits_{i=0}^{\frac{q-1}{2}} \left(q+1-2i \right)^2 & \mbox{ if } q \equiv 3 \pmod{4}. 
    \end{cases}
\]
So, when $q$ is odd, in $\cT$ there are at least 
\begin{align*}
    \frac{|\cT|}{|\Pi_{\cP}^{\cJ} \cap \cT|} \ge 
   \begin{cases} 
   \left(\frac{(q-5)(q-1)}{16}\right)^{\frac{q+1}{2}} \frac{(q^2-1)^{\frac{q-1}{2}}}{2hq^2(q+1)} & \mbox{ if }q \equiv 1 \pmod{4} \\ 
   \left(\frac{q-3}{4}\right)^{q+1} \frac{(q^2-1)^{\frac{q-1}{2}}}{2hq^2(q+1)} & \mbox{ if } q \equiv 3 \pmod{4}
   \end{cases}
\end{align*}
not equivalent line-parallelisms of $\Sigma_{\eta}$. Indeed,
\begin{align*}
    \frac{\prod_{i = 0}^{\frac{q-1}{2}} \left(q+1 - 2 i \right)^2}{2hq^2(q^2-1)(q+1)} > \frac{(q^2-1)^{\frac{q-1}{2}}}{2hq^2(q+1)},
\end{align*}
since
\begin{align*}
    \prod_{i = 0}^{\frac{q-1}{2}} \left(q+1 - 2i \right)^2 > \prod_{i = 0}^{\frac{q-1}{2}} \left((q+1)^2 - 2 (q+1) \right) = \prod_{i = 0}^{\frac{q-1}{2}} \left(q^2-1 \right) = \left(q^2-1\right)^\frac{q+1}{2}.
\end{align*}
\end{proof}

\subsection{Further comments}

In this final part we outline further properties related to $\cP$, $\eps(\cP)$ and $\Pi_{\cP}$. 
%Observe that each of the lines
%\begin{align*}
%    & s'_{u_i}: X_2 = u_i X_1, & u_i \in \cU, \\
%    & s''_{v_i}: X_3 = v_i X_1, & v_i \in \cU,
%\end{align*}
%meets $\cZ$ in $q+1$ points, and  
%\begin{align*}
%    & \{s'_{u_i} \cap \cZ \mid u_i \in \cU\}, \\
%    & \{s''_{v_i} \cap \cZ \mid v_i \in \cU\}, 
%\end{align*}
%both are partitions of $\cZ$. Moreover,
%\begin{align*}
%    |\left\{ r \in \cL_{\cP} \mid P_{u_i} \in r \right\}| & = q |s'_{u_i} \cap \eps(\cP)|, \\
%    |\left\{ r \in \cL_{\cP} \mid r \subset \pi_{v_i} \right\}| & = q |s''_{v_i} \cap \eps(\cP)|.
%\end{align*}
The set $\cZ_\alpha$ occurs as the intersection of two Hermitian curves, see for instance \cite{giuzzi2001collineation}. The linear collineation group preserving it, say $G$, has order $3(q+1)^2$, with structure $\left(C_{q+1} \times C_{q+1}\right) \rtimes S_3$, where $S_3$ is the symmetric group acting in its natural representation on the set $\{(1,0,0), (0,1,0), (0,0,1)\}$, see \cite[Theorem 2.3]{giuzzi2001collineation}. Denote by $G_1$ the stabilizer of $(1,0,0)$ in $G$ and by $H$ the subgroup of $G$ fixing each of the three points $(1,0,0)$, $(0,1,0)$, $(0,0,1)$. Then $G_1 = \langle H, \delta \rangle$, where $\delta$ is the projectivity represented by 
\begin{align*}
    \begin{pmatrix}
        1 & 0 & 0 \\
        0 & 0 & 1 \\
        0 & 1 & 0
    \end{pmatrix}, &
\end{align*}
whereas the elements of $H$ are represented by 
\begin{align}
    & \begin{pmatrix}
        1 & 0 & 0 \\
        0 & u & 0 \\
        0 & 0 & v
    \end{pmatrix}, & u,v \in \cU. \label{mat1}
\end{align}

\begin{prop}\label{prop4}
    Let $\cP$ be a good set. If $g \in G_1$, then $\eps^{-1}\left(\eps(\cP)^g\right)$ is a good set.
\end{prop}
\begin{proof}
    %Since $g \in G_1$, we have that $\cY^g = \cY$ and $\cC^g = \cC$. 
    Let $g \in G_1$. Then $\left(s_c, \cC_b\right)^g$ equals either $\left(s_\frac{cu}{v}, \cC_{buv}\right)$ or $\left(s_\frac{u}{cv}, \cC_{buv}\right)$, according as $g$ is represented by \eqref{mat1} or by
\begin{align*}
    & \begin{pmatrix}
        1 & 0 & 0 \\
        0 & 0 & u \\
        0 & v & 0
    \end{pmatrix}, & u,v \in \cU.
\end{align*}
    Hence if $c,b \in \cU$, there exist $c', b' \in \cU$ such that $s_c^g = s_{c'}$ and $\cC_b^g = \cC_{b'}$. By Proposition~\ref{prop3}, it is enough to observe that 
    \begin{align*}
        1 & = |s_c \cap \eps(\cP)|
        = |\left(s_c \cap \eps(\cP)\right)^g| 
        = |s_{c'} \cap \eps(\cP)^g|, \\
        1 & = |\cC_b \cap \eps(\cP)|
        = |\left(\cC_b \cap \eps(\cP)\right)^g| 
        = |\cC_{b'} \cap \eps(\cP)^g|.
    \end{align*}
\end{proof}

\begin{prop}
    If $g \in H$, then $\Pi_{\cP}$ and $\Pi_{\eps^{-1}\left( \eps(\cP)^g \right)}$ are equivalent.
\end{prop}
\begin{proof}
Let $g$ be represented by \eqref{mat1}. Then the projectivity of $\Sigma_{\eta}$ given by 
\begin{align*}
    & \begin{pmatrix}
        1 & 0 & 0 & 0 \\
        0 & c & 0 & 0 \\
        0 & 0 & u & 0 \\
        0 & 0 & 0 & c^q u
    \end{pmatrix}, 
\end{align*}
where $c$ is an element of $\F_{q^2}$ such that $c^{q-1} = \frac{v}{u}$, sends $\cP$ to $\eps^{-1}\left( \eps(\cP)^g \right)$.
\end{proof}
Let $\cP = \left\{\left(P_{\alpha_1 u_1}, \pi_{\alpha_1 v_1}\right), \dots, \left(P_{\alpha_{q+1} u_{q+1}}, \pi_{\alpha_{q+1} v_{q+1}}\right)\right\}$ be a good set and denote by 
\begin{align*}
    \cP^d := \eps^{-1}\left( \eps(\cP)^\delta \right) = \left\{\left(P_{\alpha_1 v_1}, \pi_{\alpha_1 u_1}\right), \dots, \left(P_{\alpha_{q+1} v_{q+1}}, \pi_{\alpha_{q+1} u_{q+1}}\right)\right\}.
\end{align*}
$\cP^d$ is also a good set, due to Proposition~\ref{prop4}, but $\Pi_{\cP}$ and $\Pi_{\cP^d}$ are not necessarily equivalent, as shown by Example~\ref{Beutelspacher} given below. 
\begin{exa}\label{Beutelspacher}
Fix $\alpha \in \cI$ and set 
\begin{align*}
\cP = \left\{ \left(P_{\alpha u}, \pi_{\alpha v}\right) \colon u \in \cU\right\}.
\end{align*}
Then it is easily seen that properties \eqref{property1}, \eqref{property2} are satisfied
%\begin{align*}
%    \left(P_{\alpha u_i}, \pi_{\alpha v_1}\right), \left(P_{\alpha u}, \pi_{\frac{u v_1}{u_i}} \right) \in \cP \implies u = u_i, \\
%    \left(P_{u_i}, \pi_{v_1}\right), \left(P_{\frac{u_i}{u}}, \pi_{u v_1}\right) \in \cP \implies u = 1,
%\end{align*}
and hence $\cP$ is a good set. In this case $\cL_{\cP}$ consists of the $q^2+q$ lines contained in the plane $\pi_{\alpha v}$, intersecting $\Sigma_{\alpha} \cap \pi_{\alpha v}$ in a Baer subline and distinct from $r_{U_1}$. The line-parallelism of $\Sigma_{\eta}$ obtained from $\cP$ via Theorem~\ref{main} coincides with the one constructed by Beutelspacher in \cite{beutelspacher1974parallelisms}. On the other hand, 
\begin{align*}
\cP^d = \left\{ \left(P_{\alpha v}, \pi_{\alpha u}\right) \colon u \in \cU\right\}
\end{align*}
is also a good set and $\cL_{\cP^d}$ consists of the $q^2+q$ lines through the point $P_{\alpha v}$, distinct from $r_{U_1}$, and intersecting $\Sigma_{\alpha}$ in a Baer subline. Therefore $\Pi_{\cP}$ and $\Pi_{\cP^d}$ are not equivalent.
\end{exa}

%\begin{exa}\label{new}
%Assume $q \not\equiv -1 \pmod{3}$, and let 
%\begin{align*}
%\cP = \left\{ \left(P_{u_1}, \pi_{u_1^2}\right), \left(P_{u_2}, \pi_{u_2^2}\right)\dots, \left(P_{u_{q+1}}, \pi_{u_{q+1}^2}\right) \right\}.
%\end{align*}
%Then 
%\begin{align*}
%    & \left(P_{u_i}, \pi_{u_i^2}\right), \left(P_{u}, \pi_{\frac{u u_i^2}{u_i}} \right) \in \cP \implies u = u_i, \\
%    & \left(P_{u_i}, \pi_{u_i^2}\right), \left(P_{\frac{u_i}{u}}, \pi_{u u_i^2}\right) \in \cP \implies u^3 = u^{q+1} = 1 \implies u = 1.
%\end{align*}
%In this case the lines of $\cL_{\cP}$ are on the $\frac{q+1}{\gcd(q-1,2)}$ planes $\pi_{v^2}$, $v \in \cU$, and each of them contains the lines of the $\gcd(q-1, 2)$ Baer subpencils $p\left(P_{v}, \pi_{v^2}\right)$, $p\left(P_{-v}, \pi_{v^2}\right)$, distinct from $r_{U_1}$.
%\end{exa}

\bigskip
{\footnotesize
\noindent\textit{Acknowledgments.}
%This research was partly supported by ...
The research was supported by the Italian National Group for Algebraic and Geometric Structures and their Applications (GNSAGA--INdAM). The second author research was supported by the INdAM - GNSAGA Project \emph{Tensors over finite fields and their applications}, number E53C23001670001 and by Bando Galileo 2024 – G24-216. The first author would like to dedicate this paper to the memory of his father.}

\bibliographystyle{abbrv}
\bibliography{biblio}

\end{document}